\ifdefined\pdfoutput\pdfoutput=1\fi
\documentclass[11pt]{article}

\usepackage{iftex}
\ifPDFTeX
  \usepackage[utf8]{inputenc}
  \IfFileExists{lmodern.sty}{\usepackage[T1]{fontenc}\usepackage{lmodern}}{\PassOptionsToPackage{expansion=false}{microtype}}
\else
  \usepackage{fontspec}
\fi
\usepackage{fvextra}
\usepackage{microtype}
\usepackage{amsmath,amssymb,amsthm,mathtools}
\usepackage{xcolor}
\usepackage[margin=1.08in]{geometry}
\usepackage{enumitem}
\usepackage[
  colorlinks=true,
  linkcolor=black,
  citecolor=black,
  urlcolor=black,
  pdftitle={Uniform convergence of canonical discrete uniformization on equilateral approximations},
  pdfauthor={Aijin Lin and Chunlei Liu}
]{hyperref}

\setlist[enumerate]{leftmargin=2.2em,itemsep=2pt,topsep=4pt}
\setlist[itemize]{leftmargin=2em,itemsep=2pt,topsep=4pt}
\allowdisplaybreaks
\numberwithin{equation}{section}
\theoremstyle{plain}
\newtheorem{theorem}{Theorem}[section]            % label kind: thm
\newtheorem{proposition}[theorem]{Proposition}    % label kind: prop
\newtheorem{lemma}[theorem]{Lemma}                % label kind: lem
\theoremstyle{definition}
\newtheorem{definition}[theorem]{Definition}      % label kind: def
\newtheorem{setup}[theorem]{Set-up}               % label kind: setup

\theoremstyle{remark}
\title{Uniform convergence of canonical discrete uniformization on equilateral approximations}
\author{Aijin Lin and Chunlei Liu}
\date{}

\providecommand{\CC}{\mathbb C} % complex plane
\providecommand{\RR}{\mathbb R} % real line
\providecommand{\ZZ}{\mathbb Z} % integers
\providecommand{\Chat}{\widehat{\mathbb C}} % ideal Riemann sphere
\providecommand{\HH}{\mathbb H} % hyperbolic space
\providecommand{\Tri}{\Delta} % closed unit equilateral reference triangle
\providecommand{\om}{\omega} % exp(pi i/3)
\providecommand{\Lam}{\Lambda} % unit triangular lattice
\providecommand{\Sig}{\Sigma_{\Lambda}} % complete lattice dome
\providecommand{\CH}{\operatorname{CH}} % hyperbolic convex hull
\providecommand{\dist}{\operatorname{dist}} % metric distance, with indicated metric
\providecommand{\diam}{\operatorname{diam}} % diameter
\providecommand{\Area}{\operatorname{Area}} % Euclidean area
\providecommand{\id}{\operatorname{id}} % identity map
\providecommand{\Mod}{\operatorname{Mod}} % conformal modulus
\providecommand{\dd}{\,\mathrm d} % differential in integrals
\providecommand{\bpartial}{\overline\partial} % antiholomorphic derivative
\begin{document}
\maketitle
\begin{abstract}
We prove uniform convergence of the canonical discrete uniformization
maps for every sequence of equilateral polygonal disks whose mesh tends
to zero and whose marked Jordan boundaries converge uniformly.
The limit is the normalized Riemann map, with convergence on the entire
closed reference triangle. This proves Conjecture 7.3 of Luo, Sun and Wu
for the original construction with Delaunay flips.
The proof combines closed-domain compactness, marked lattice limits,
a packing and resistance estimate for logarithmic cusp scales, and an
envelope comparison principle, followed by a local affine energy argument.
\end{abstract}

\tableofcontents
\section{Introduction}\label{sec:01-introduction}

Discrete conformal geometry assigns conformal data to finite polyhedral
surfaces while allowing their triangulations to change. A basic question
is whether the resulting canonical maps recover ordinary conformal maps
as the mesh tends to zero. We study the construction of Luo, Sun and
Wu~\cite[\S7.1, pp.~981--983]{luo2022} for planar polygonal disks
with an equilateral triangulation and three ordered boundary
marks.

Double such a disk along its boundary. In its discrete conformal class,
prescribe curvature \(4\pi/3\) at the three distinguished vertices and zero
at every other vertex, and normalize the doubled area to \(\sqrt3/2\).
The resulting metric is the double of the unit equilateral triangle
\(\Tri=\operatorname{conv}\{0,1,e^{\pi i/3}\}\).
The marked isometry of the associated ideal hyperbolic surfaces,
composed with their vertical projections, gives a map from \(\Tri\)
to the original closed disk. This is the canonical piecewise-projective
map of \cite[pp.~982--983]{luo2022}; its construction permits Delaunay
flips. Convergence would make the conformal limit independent of the
particular equilateral polygonal approximation.

\paragraph{Previous work.}
Circle packings provided an early geometric approach to conformal
approximation. Rodin and Sullivan~\cite[\S\S4--5]{rodin1987} combined ring,
length--area and hexagonal-packing estimates to prove convergence.
He and Schramm~\cite[Theorem~1.1]{he1996} established packing convergence
under their boundary correspondence and shrinking-disk hypotheses
without a uniform bounded-valence assumption. These results explain the
value of separating compactness from microscopic rigidity. The cusp
packing estimates for the canonical maps studied here are proved below
from their decorated hyperbolic geometry.

The finite theory supplies coordinates for this construction.
Luo~\cite[\S1.2 and Theorem~1.2]{luo2004} developed a variational
curvature-flow description of Euclidean vertex scaling.
Bobenko, Pinkall and Springborn~\cite[Definition~2.6.1 and
Proposition~5.3.2]{bobenko2015} related discrete conformality,
piecewise-projective maps and ideal hyperbolic geometry.
The boundary-value formulation of Bobenko, Sechelmann and
Springborn~\cite[Theorems~3.7, 3.9 and Proposition~3.16]{bobenko2016}
gives finite variational problems with the stated uniqueness
qualifications and scale freedoms.
Gu, Luo, Sun and Wu~\cite[Theorem~4.3 and Corollaries~4.7--4.8]{gu2018}
identify finite marked Euclidean polyhedral metrics with decorated
hyperbolic data across Delaunay cells. Their Euclidean uniformization
theorem is an essential input below.
Gu, Guo, Luo, Sun and Wu~\cite[Theorem~3 and Corollary~4]{gu2018a}
establish the hyperbolic counterpart, including surgery along the flow.

Penner's lambda-length coordinates and diagonal-change rule
\cite[Proposition~2.6(a) and Theorem~3.1]{penner1987}
provide the finite-type coordinate background.
Springborn's variational treatment of ideal polyhedra and discrete
uniformization~\cite[Theorem~1.1 and Proposition~7.7]{springborn2020}
supplies curvature gradients and the relevant fiber Hessian; its
convexity assertion concerns those fibers.
Bobenko and Springborn~\cite[Definition~16 and Proposition~17]{bobenko2007}
explain nonnegative cotangent weights for intrinsic Delaunay Laplacians.
These finite structures alone do not give uniform control of infinite
limiting cusp stars when target cells change.

Several convergence results use geometric assumptions adapted to their
discretization. Wu and Zhu~\cite[Definition~1.3 and
Theorems~1.4--1.5]{wu2024} obtain quantitative conformal-factor convergence
for regular geodesic triangulations of smooth closed surfaces.
Luo, Wu and Zhu~\cite[Definition~1.2 and Theorem~1.3]{luo2024a}
treat normalized genus-zero convergence on a smooth sphere with regular
triangulations.
Gurel-Gurevich, Jerison and Nachmias~\cite[Theorem~1.1]{gurelgurevich2020}
prove a logarithmic Dirichlet error estimate for orthodiagonal maps under
weak local mesh geometry and their boundary-data hypotheses.
Binder and Pechersky~\cite[Theorem~2.15]{binder2025}
prove convergence of orthodiagonal tiling maps and extremal lengths for
good interior approximations with four marked prime ends.
Glickenstein and Sidbury~\cite[Theorem~1 and \S6]{glickenstein2025}
give local distortion criteria for convergence of barycentric maps
between corresponding triangulations. Their map classes and geometric
hypotheses must be kept in mind when comparing with the present problem.

A further relevant line of work concerns infinite rigidity.
Wu, Gu and Sun~\cite[Theorem~1]{wu2015} prove planar hexagonal rigidity
under a uniform acute-angle hypothesis.
Dai, Ge and Ma~\cite[Theorem~1.4]{dai2022} obtain Delaunay rigidity on
standard hexagonal combinatorics using a maximum principle with a
hyperbolic interpretation.
Dai and Wu~\cite[Theorem~1.3]{dai2024} treat planar Delaunay rigidity on
a common triangulation under their nondegeneracy assumptions.
Luo, Xu and Zhang~\cite[Theorem~1.3, version~2]{luo2025}
prove inversive-distance packing rigidity on fixed hexagonal
combinatorics with specified regular distance data.
Rivin's finite ideal-polyhedron realization and uniqueness theorems
\cite[Theorems~1.1--1.3]{rivin1994} give a geometric predecessor.
Our limiting problem must also keep every cusp marking, control the
possible lifts of incident ideal arcs, and preserve the limiting angle
identities when finite Delaunay cells vary.

Convex geometry supplies compactness inputs.
Luo and Wu~\cite[Theorems~4.1, 5.1 and 11.2]{luo2024}
give the Euclidean boundary-area bound in the Poincar\'e ball,
intrinsic-distance convergence, and conformal length and area
structures for nonsmooth convex boundaries.
The dimension-sensitive boundary limits studied by Alesker, Katz and
Prosanov~\cite[Proposition~2.5 and Claim~2.7]{alesker2023}
retain a doubled surface in codimension one.
Prosanov~\cite[Lemma~3.28 and the proof of Lemma~5.1]{prosanov2024}
gives related intrinsic-convergence arguments for polyhedral hyperbolic
boundaries. These comparisons show that one must keep both sheets
in a planar degeneration and exclude shortcuts through artificial
truncation caps.

Energy and modulus offer complementary tools.
Ha\"issinsky~\cite[Proposition~2.3 and Appendix~B]{haissinsky2009}
compares graph moduli under bounded-detour hypotheses and analytic
moduli for the specified quasi-packing covers.
Doyle and Snell~\cite[\S1.3.5 and \S\S2.2.2--2.2.3]{doyle1984}
develop electrical variational principles relating energy and
resistance. Below we prove the quantitative resistance estimate
directly on the original unit triangular lattice and combine it with
packing inequalities for the actual canonical geometry.

The issue in \cite[Conjecture~7.3]{luo2022} is therefore convergence
for arbitrary equilateral approximating sequences and the actual
canonical maps. Compactness must extend to a possibly nonsmooth Jordan
boundary. At the microscopic scale, intrinsic isometry must yield marked
planar similarity without assuming that the target retains the original
cells. That conclusion must then identify the limit of the full maps and
of the entire sequence.

\paragraph{The main result.}
We prove Conjecture~7.3 of Luo--Sun--Wu in its original qualitative form.
A marked equilateral disk \((P,T,h;a,b,c)\) is a finite planar polygonal
closed disk with an equilateral triangulation of common side length
\(h>0\) and three distinct boundary vertices in positive cyclic order.
Uniform convergence of Jordan boundaries means that homeomorphic circle
parametrizations converge uniformly to a homeomorphic parametrization of
the limiting boundary. Use the canonical map defined above, taking the
corners \(0,1,e^{\pi i/3}\) to the ordered marked vertices.

\begin{theorem}\label{thm:introduction-main}
Let $\Omega\subset\CC$ be a bounded Jordan domain with three distinct
positively ordered marked boundary points $p,q,r$. Let
$(P_n,T_n,h_n;p_n,q_n,r_n)$ be any sequence of marked equilateral disks
such that $h_n\to0$, their Jordan boundaries converge uniformly to
$\partial\Omega$, and $(p_n,q_n,r_n)\to(p,q,r)$.
Then the canonical discrete uniformization maps
$f_n:\Tri\to P_n$ converge uniformly on the entire closed triangle
$\Tri$ to the normalized Riemann map
$\phi:\Tri\to\overline\Omega$ with the prescribed corner values.
\end{theorem}

Theorem~\ref{thm:introduction-main} is Theorem~\ref{thm:main}, proved in
Section~\ref{sec:08-main-proof}. The conclusion includes every boundary
point and all three corners. Any approximating sequence satisfying the
stated assumptions is allowed, and the canonical construction retains
its Delaunay flips. The theorem asserts no rate depending only on mesh size.

\paragraph{The proof.}
First we obtain compactness of the actual maps and their inverses on
closed domains. A conformal correction of each equilateral source face
moves points by at most one mesh length. Periodic reflection of the
complete target surface, the convex-area bound, and annulus extremal
length give two-sided continuity estimates and homeomorphic limits.

The microscopic argument uses a quantitative invariant.
Canonical cusp scales give disjoint planar disks and scale comparison
on every original edge. Packing controls a cutoff energy on the original
lattice; its Fourier resistance estimate gives a uniform two-point bound
\[
 |u_v-u_w|\le A\log\bigl(N(v,w)+2\bigr).
\]
Here \(u_v\) is the logarithmic cusp scale, \(A\) is an absolute constant,
and \(N(v,w)=|m|+|n|\) when \(v-w=m+n e^{\pi i/3}\).
The bound is uniform in the choice of the root. It makes limiting horocycle envelopes
positive and finite and restricts their active lifts to a fixed finite
family. Marked intrinsic ball identifications also control winding
classes. Thus the full \(2\pi\) angle identities survive the limit.

A translated increment of the logarithmic factors attains its supremum
after extraction. The envelope angle formula and a strong comparison
principle propagate that maximum; a nonzero constant increment would
contradict logarithmic growth. The limiting scales are therefore
constant. Strict persistence of the marked equilateral cells gives
convergence of the actual canonical maps on closed microscopic faces.
Finally, affine interpolation provides an analytic comparison:
a local degree argument controls multiplicity and energy, and the
antiholomorphic derivative tends to zero. Each homeomorphic limit is
then the normalized Riemann map.

\paragraph{Organization.}
Section~\ref{sec:02-preliminaries} fixes the marked surfaces, maps,
decorations and classical tools.
Section~\ref{sec:03-source-scales} proves source correction, reflection
covering and canonical scale estimates.
Section~\ref{sec:04-compactness} proves closed-domain compactness.
Section~\ref{sec:05-lattice-limits} constructs whole marked lattice limits.
Section~\ref{sec:06-logarithmic-rigidity} proves packing growth, envelope
comparison and logarithmic Liouville rigidity.
Section~\ref{sec:07-microscopic-similarity} proves marked microscopic
similarity. Section~\ref{sec:08-main-proof} gives the complete proof of
the main theorem.

\section{Preliminaries}\label{sec:02-preliminaries}
\begin{setup}\label{setup:canonical}
Set $\om=e^{\pi i/3}$, $\Lam=\ZZ+\ZZ\om$, and
$\Tri=\operatorname{conv}\{0,1,\om\}$. The set $\Tri$ includes its boundary.
A marked equilateral disk $(P,T,h;a,b,c)$ is a finite planar polygonal
closed disk with an equilateral triangulation of common side length $h>0$
and three distinct boundary vertices in positive cyclic order.
The doubled metric is the Euclidean cone metric on the marked sphere
obtained by gluing two copies of $P$ along their boundaries.
Write $V(P)$ for the vertex set and identify a vertex label with its planar position.
A PL metric on a marked surface $(S,V)$ means a flat cone metric whose cone
points are contained in $V$; marked equivalence fixes every point of $V$.
For the marked disk, the canonical discrete conformal metric has curvature $4\pi/3$ at $a,b,c$,
zero at all other vertices, and doubled area $\sqrt3/2$.
It is the metric double of $\Tri$.
\end{setup}

Discrete conformality means vertex scaling through Delaunay triangulations,
allowing changes of triangulation as in \cite[Definition~7.1]{luo2022}.
Replace a Euclidean Delaunay face by the ideal hyperbolic triangle with
the same vertices in upper half-space, and glue along the extensions of
the Euclidean edge identifications. This gives a complete punctured hyperbolic
surface. Restrict to the corresponding halves of these punctured hyperbolic surfaces,
denoted $H_P$ and $H_\Tri$. Let $p_P:H_P\to P$ and
$p_\Tri:H_\Tri\to\Tri$ be the vertical
projections, and let $I:H_P\to H_\Tri$ be the canonical marked intrinsic
isometry. The maps extend to the filled cusp points and closed disks, with
\[
 g=p_\Tri\circ I\circ p_P^{-1},\qquad
 f=p_P\circ I^{-1}\circ p_\Tri^{-1}=g^{-1}.
\]
The corners correspond to $a,b,c$.
These are the piecewise-projective maps of \cite[\S7.1]{luo2022}.
They are not defined by affine interpolation on the original triangulation.

For a closed ideal set $Y\subset\Chat$, $\CH(Y)$ denotes its hyperbolic
convex hull and $\partial\CH(Y)$ its intrinsic boundary. A two-dimensional
hull is interpreted as the metric double of its convex planar surface.
The whole lattice surface is
$\Sig=\partial\CH(\Lam\cup\{\infty\})$.
Its standard decoration has cusp circumference $6$ at every lattice label;
corner horocycles in each unit ideal face have corner arc length $1$.
Adding all finite cusps and the residual end at infinity to $\Sig$ gives
its filled source compactification, denoted $\widehat\Sig$; the lattice
facewise conformal construction identifies it with $\Chat$.
The source coordinate distance is
$N(v,w)=|m|+|n|$ when $v-w=m+n\om$.

A decoration specifies positive intrinsic horocycle circumferences, which
may be formally extended beyond embedded cusp neighborhoods. The lambda
length of an ideal arc is $\exp(D/2)$, where $D$ is the signed intrinsic
distance between its endpoint horocycles. In an upper-half-plane chart,
a finite-center horocycle with center $a\in\RR$ and Euclidean diameter
$d>0$ has signed Busemann function
$B_{a,d}(x,y)=\log(((x-a)^2+y^2)/(dy))$.
The root horocycle $y=1$ has function $-\log y$.
A genuine Voronoi neighbor shares an arc of positive length; optional
polygon diagonals are excluded from this terminology.

Uniform convergence of Jordan boundaries is understood in the sense of
\cite[p.979]{luo2022}: homeomorphic circle parametrizations converge
uniformly to a homeomorphic parametrization of the limiting boundary.
All compact-containment statements $K\Subset\Omega$
refer to the physical planar domain. We use
$\partial L=(L_x-iL_y)/2$ and $\bpartial L=(L_x+iL_y)/2$ for complex-valued
maps; $|DL|$ is the Hilbert--Schmidt norm and
$J_L=|\partial L|^2-|\bpartial L|^2$.

For the classical topological tools used below, our references are
\cite[Propositions~1.30, 1.33--1.34, pp.~60--62]{hatcher2002} for
covering lifts and their uniqueness,
\cite[Proposition~2B.1 and the Schoenflies paragraph, p.~169]{hatcher2002}
for Jordan separation and disks, and
\cite[\S2.2, pp.~134--136, Proposition~2.30]{hatcher2002} for
homotopy invariance of degree and the sum of local degrees.
Sobolev slicing and reflection use the ACL (absolutely continuous on
lines) characterization in
\cite[\S2.6, Theorem~2.36, pp.~55--59; Example~2.39, p.~59;
Remark~2.41, p.~60]{kinnunen2026}; the matching traces and logarithmic
cutoffs at isolated vertices are supplied in the proof below.
For compactness we use the classical Arzel\`a--Ascoli theorem and
compactness of closed subsets in a compact metric space, in the forms
recorded in \cite[Theorem~5.8, pp.~20--21, and \S5, p.~16]{luo2024}.
The plane and sphere automorphism facts and the local uniform limit
theorem for holomorphic functions are recorded in
\cite[Theorem~2.13, Corollary~2.14, and Theorems~1.4, 1.22]{mcmullen2025}.

\section{Source geometry and canonical cusp scales}\label{sec:03-source-scales}

We keep the notation for marked equilateral disks and their canonical
maps. We first identify the conformal structure of the source, then unfold
the target and compare its cusp decorations. These comparisons give the
scale bounds on original edges, including edges removed by Delaunay flips.
The finite input is \cite[Theorem~7.2 and pp.~981--982]{luo2022}, with
\cite[Theorem~1.2, Proposition~4.1, Theorem~4.3 and Corollary~4.8]{gu2018}
for the decorated correspondence and its ideal-surface realization.

\begin{theorem}[Finite discrete uniformization]\label{thm:finite-uniformization}
Let $S$ be a closed connected surface, let $V\subset S$ be finite and nonempty,
and let $(S,V)$ be equipped with a Euclidean
polyhedral metric $d$. For a function $K:V\to(-\infty,2\pi)$ with
$\sum_{v\in V}K(v)=2\pi\chi(S)$, there is a metric discretely conformal
to $d$ with curvature $K$, unique up to scale and a marked isometry
homotopic to the identity. If $\chi(S\setminus V)<0$, the marked PL
metrics correspond to decorated complete finite-area cusp hyperbolic
metrics: on every Delaunay cell, Euclidean lengths are lambda lengths and
the Euclidean cells correspond to the decorated Voronoi/Delaunay cells.
These identifications agree on overlaps under flips. The underlying marked
hyperbolic metric is unchanged exactly within a discrete conformal class.
\end{theorem}

The doubled marked sphere has at least three punctures, and hence negative
Euler characteristic once the punctures are removed. Thus the source and
target are two decorations of one marked hyperbolic surface. The
canonical map uses this marked identification and the vertical projections,
as in \cite[pp.~982--983]{luo2022}.

We record the complex-analysis results used to glue the conformal face
maps. The first combines \cite[Theorems~4.1 and~4.5]{mcmullen2025}
with the disk-automorphism normalization in
\cite[Theorems~2.22--2.24]{mcmullen2025}.

\begin{theorem}[Riemann mapping and Jordan extension]\label{thm:jordan-riemann}
A bounded Jordan domain is conformally equivalent to the unit disk, and
a conformal equivalence extends to a homeomorphism of the closures.
Three distinct compatible boundary values determine the normalized
conformal equivalence uniquely.
\end{theorem}

This determines the conformal map of each closed ideal face from its
vertex values.
The local extension across a face edge is
\cite[Theorem~1.54 and the following analytic-arc remark]{mcmullen2025}.

\begin{theorem}[Schwarz reflection]\label{thm:reflection}
A holomorphic function on one side of a real interval that extends
continuously to the interval with real boundary values extends holomorphically
across it by complex conjugation. The corresponding local assertion holds
across real-analytic boundary arcs after conformal coordinate changes.
\end{theorem}

The reflection statement identifies the holomorphic extensions from
adjacent faces. At a filled cusp we use
\cite[Theorem~1.31]{mcmullen2025}.

\begin{theorem}[Removable punctures]\label{thm:removable}
A holomorphic function on a punctured disk that is bounded near its
puncture extends holomorphically over the puncture.
\end{theorem}

Boundedness in a local coordinate therefore suffices to fill the cusp
conformally once the map is a homeomorphism.

We also use the classical annulus uniformization and extremal-length
duality in \cite[Theorems~6.1 and~6.5, pp.~59,~63]{mcmullen2021}.

\begin{theorem}[Annulus modulus]\label{thm:annulus-duality}
Let $A$ be a plane annulus bounded by two disjoint Jordan curves.
There are $0<r<R<\infty$ and a conformal equivalence of $A$ with
$\{r<|z|<R\}$. The extremal length of curves joining the boundary
components is $\Mod(A)=(2\pi)^{-1}\log(R/r)$; the extremal length of
essential separating closed curves is $1/\Mod(A)$.
\end{theorem}

The reciprocal formula allows us to compare annuli by inclusion of their
families of separating curves.

\begin{proposition}\label{prop:source-correction}
For every marked equilateral disk, the cusp-filled associated hyperbolic
surface admits a face-preserving conformal identification with the doubled
Euclidean cone surface, commuting with the doubling involution.
On either copy the induced correction $\kappa:P\to P$ fixes all vertices,
preserves every closed face and edge, and satisfies
$\|\kappa-\id\|_\infty\le h$.
If $h<r<R$, $R-r>2h$, and $\overline B(z,R+h)\subset\operatorname{int}P$,
the conformal modulus $M$ of the cusp-filled lift of
$\{r<|x-z|<R\}$ satisfies
\[
 \frac1{2\pi}\log\frac{R-h}{r+h}\le M\le
 \frac1{2\pi}\log\frac{R+h}{r-h}.
\]
\end{proposition}

The correction changes a point only within its source face. Consequently
round annuli in the source have asymptotically the same modulus as their
filled hyperbolic lifts when the mesh is small compared with their radii.

\begin{proof}
We construct a normalized conformal map on each face, check that the maps
extend across edges and cusps, and compare the resulting annulus with two
round annuli. Write $S_P$ for the Euclidean double and $H_P^{\mathrm d}$
for its associated punctured hyperbolic double. Let
$p_P^{\mathrm d}:\overline{H_P^{\mathrm d}}\to S_P$ be vertical projection
on its conformal compactification, extended at the vertices. Its
restriction to the chosen half is the map $p_P$ from the preliminaries.

\emph{The face maps agree on edges.}
The compactification of an ideal triangle in the disk model is a Jordan
disk. By Theorem~\ref{thm:jordan-riemann}, there is a unique normalized
conformal map to its equilateral Euclidean counterpart that extends
homeomorphically to the boundary and takes the three ideal vertices to
their designated vertices. Denote this map by $\psi$.
Every permutation of the three vertices
is realized by an isometry of each triangle. Conjugating $\psi$ by a
matching pair of these isometries gives another conformal map with the
same vertex values, also when both isometries reverse orientation.
Uniqueness therefore makes $\psi$ equivariant under all the paired
symmetries.

On the ideal edge above a Euclidean side of length $h$, the symmetry
interchanging its endpoints fixes the top of the semicircle, at height
$h/2$. The original edge identification preserves this distinguished
point and the prescribed endpoints. These data determine the isometry
of the ideal edge. The equivariance just proved therefore makes the
traces of the two incident face maps agree. The face maps glue to a
homeomorphism $Q:\overline{H_P^{\mathrm d}}\to S_P$ with the prescribed
vertex values.

\emph{The glued map is conformal at edges and vertices.}
Develop two incident ideal faces into the hyperbolic plane on opposite
sides of their shared geodesic. Develop their Euclidean counterparts on
opposite sides of the shared straight edge. The two reflections interchange the face maps. Theorem~\ref{thm:reflection} makes their common
extension holomorphic across the edge. Since the extension is locally
one-to-one, its derivative there does not vanish.

At each ideal vertex, take the congruent horocyclic sectors tangent at the
distinguished edge points, and shrink them by one common factor. Their
endpoints match along every glued edge. The sectors around a vertex
form a horodisk modulo a nonzero parabolic translation. If $\zeta$ is a
coordinate in that horodisk and $\ell>0$ is its translation period, then
$\exp(2\pi i\zeta/\ell)$ is a punctured-disk coordinate. On the Euclidean
side, a cone of total angle $\alpha$ has local complex coordinate
$\xi^{2\pi/\alpha}$, with $\xi$ taken in a developed sector. In these
coordinates $Q$ is holomorphic off the puncture and tends to the vertex.
It is bounded near the puncture, so Theorem~\ref{thm:removable} extends
it holomorphically there. Its extension is a homeomorphism, so its power
series has order one and it is conformal at the filled point.

The construction is unchanged by interchanging the two copies of $P$.
Hence $Q$ commutes with doubling. Both $Q$ and $p_P^{\mathrm d}$ carry each closed
ideal face onto its corresponding Euclidean face and agree on vertices.
On either half, $\kappa=Qp_P^{-1}$ therefore preserves closed faces and
edges and fixes vertices. A source face has diameter $h$, giving
\[
 |\kappa(x)-x|\le h,
 \qquad |\kappa^{-1}(x)-x|\le h.
\]

\emph{The displacement bounds control modulus.}
Put $A=\{x:r<|x-z|<R\}$ and $U_s=\kappa(B(z,s))$ for $s=r,R$.
The bounds for $\kappa$ and its inverse give
\[
 B(z,s-h)\subset U_s\subset B(z,s+h).
\]
For example, if $|y-z|<s-h$, then
$|\kappa^{-1}(y)-z|<s$. The hypotheses ensure that all these disks lie
in the source interior. The corresponding inclusions of closed disks
also hold. It follows that
\[
 \{r+h<|x-z|<R-h\}
 \subset U_R\setminus\overline U_r
 \subset\{r-h<|x-z|<R+h\},
\]
and each inclusion is an essential annulus inclusion. The middle annulus
is $Q(p_P^{-1}(A))$, with its cusp points filled. Its modulus is $M$ by
conformality of $Q$. Modulus increases under essential annulus inclusion:
the family of core curves of the smaller annulus is a subfamily of that
of the larger annulus, so its extremal length is larger. By
Theorem~\ref{thm:annulus-duality}, this extremal length is the reciprocal
modulus. Applying this observation to the two inclusions and using the
round-annulus formula from that theorem proves both bounds.
\end{proof}

\begin{proposition}\label{prop:reflection-cover}
For a marked equilateral disk, reflect its canonical target vertex set
in the three sides of $\Tri$ to obtain the locally finite periodic set
$X\subset\CC$. The dome $\partial\CH(X\cup\{\infty\})$ is the reflection
cover of the doubled target punctured hyperbolic surface.
The canonical intrinsic map embeds the punctured open source half in this
cover, with peripheral degree one at original interior vertices; the cover
has degree three at the three marked corners of the doubled quotient.
Applying a direct upper-half-space similarity that makes the source mesh
one identifies every interior original patch with the corresponding patch
of $\Sig$; this similarity preserves hyperbolic distances.
If the Euclidean distance of a chosen face to the source boundary divided
by $h$ tends to infinity along a sequence, write $\Sigma_n=\partial\CH(X_n\cup\{\infty\})$, choose a fixed lattice
face incenter $o$, and use consistent lifted maps $I_n$. For every $R>0$,
eventually $I_n$ is an onto isometry from the closed $R$-ball about $o$
in $\Sig$ to the closed $R$-ball about $I_n(o)$ in the complete surface
$\Sigma_n$, with both distances inherited from the complete surfaces.
\end{proposition}

The last assertion compares distances in the complete surfaces. It also
preserves the peripheral loops around the source vertices in a fixed ball.

\begin{proof}
We realize the target cover by its Delaunay polygons, compute its
monodromy, and then use a projection estimate to keep short paths inside
an original source patch.

\emph{Unfolding gives the target dome.}
Let $G$ be the group generated by reflection in the three side lines of
$\Tri$, and let $\Gamma$ be its orientation-preserving subgroup. The
reflected triangles tile the plane. One triangle of each orientation
forms a fundamental domain for $\Gamma$, and the quotient is the metric
double of $\Tri$. At a tile vertex the stabilizer in $\Gamma$ has order
three: six angles of size $\pi/3$ upstairs project to the two such angles
in the double. Every other point has trivial stabilizer. Thus the
quotient map is an ordinary covering after the marked vertices and all
their preimages are removed.

Let $W$ be the marked vertex set on the doubled target. It is invariant
under the doubling involution. Its full inverse image is $X$.
The group $G$ has a translation lattice of finite index, so $X$ is a
finite union of lattice translates. It is locally finite, has positive
minimum separation, and contains every vertex of the unit equilateral
tiling. In particular its covering radius is at most $1$.

The planar Delaunay cells of $X$ are the maximal convex polygons of
points on circles with empty open disks. Every empty disk has radius
at most $1$. Thus the vertices of a cell meeting a fixed compact set
lie within distance $2$ of that set, and local finiteness of $X$ gives
local finiteness of the decomposition. For a cell with
circumcenter $\zeta$ and circumradius $\rho$, its ideal polygon lies on
\[
 t^2+|x-\zeta|^2=\rho^2,\qquad t>0,
\]
and projects onto the Euclidean cell. Common sides have identical
semicircular lifts. These polygons form the boundary of
$\CH(X\cup\{\infty\})$: the complements of the empty hemispheres are
supporting hyperbolic half-spaces containing the ideal set. More
explicitly, lift each $x\in X$ to $(x,|x|^2)$ in Euclidean three-space.
The lower convex faces project to the Delaunay cells. On the cell above,
the affine lower face has height
\[
 L(x)=2\zeta\mathbin{\cdot}x+\rho^2-|\zeta|^2,
 \qquad t^2=L(x)-|x|^2.
\]
Thus the supporting faces give a vertical graph over $\CC\setminus X$,
and exhaust its convex-hull boundary. Adding infinity introduces no
further finite face, since $X$ is relatively dense in the entire plane.

This construction commutes with $G$. The action of $\Gamma$ on the dome
is free: any nontrivial stabilizer in the plane is a rotation about a tile
vertex, and its vertical fixed axis does not meet the dome. The quotient
ideal polygons, with their edge identifications, give precisely the
hyperbolic surface associated to the doubled target. A diagonal in a
cyclic polygon only subdivides one hyperbolic plane, so the quotient
metric is independent of that choice. This agrees with the marked
identification in Theorem~\ref{thm:finite-uniformization}.

\emph{The marked map has trivial monodromy on a source half.}
Let $I^{\mathrm d}$ be the canonical marked intrinsic isometry between
the hyperbolic doubles, whose restriction to the chosen half is $I$.
It intertwines their doubling involutions. Indeed, conjugating by the
two involutions gives an isometry in the same marked homotopy class,
so uniqueness gives the same map. On the dome, reflection in a target
side fixes exactly the graph above that side. Consequently the chosen source half maps to the
half represented by the graph over $\Tri$.

Around a noncorner marked point of the target double, the monodromy of
the cover is the identity. Around a marked corner it is a rotation of
order three. Pull the cover back by $I^{\mathrm d}$. The fundamental group of the
punctured open source half is generated by loops around its interior
vertices. Each such vertex is marked and has a noncorner target image,
so each generator has trivial monodromy. Once a lift is chosen at one
point, the entire punctured half therefore lifts injectively. Adding
its open boundary edges by continuity gives the same lift there. In
particular a peripheral loop around an interior source vertex winds
once around its single lifted target cusp. The corresponding complete
cusp neighborhood also lifts with degree one. The degree of a full
doubled corner cusp in the reflection cover is three; that cusp is not
a peripheral loop in the punctured half.

The pullback of the full target cover, equipped with the lifted source
metric, is globally isometric to the target dome. We may therefore
check metric balls on that complete source cover.

\emph{Deep source patches contain every short path.}
Choose $x_0\in\CC$ and $\vartheta\in\RR$ so that the direct similarity
$x\mapsto e^{i\vartheta}(x-x_0)/h$ takes the chosen face to a fixed unit
lattice face. Its upper-half-space extension is
\[
 (x,t)\longmapsto
 \left(\frac{e^{i\vartheta}(x-x_0)}h,\frac th\right).
\]
The Euclidean length element and the height are both divided by $h$,
so this map preserves hyperbolic distances. An ideal face above a unit
equilateral triangle lies on a hemisphere of radius $1/\sqrt3$.
Its intrinsic metric and its edge gluing are exactly those of the
corresponding face of $\Sig$. The injective lift just constructed thus
identifies every interior source face patch with the corresponding
lattice patch, including complete cusps at vertices whose full stars
belong to the patch.

Let $p_{\mathrm{loc}}$ denote vertical projection on such a patch.
For any rectifiable path $\gamma$ in it, the height bound gives
\begin{equation}\label{eq:source-scales-projection}
 \operatorname{length}_{\CC}(p_{\mathrm{loc}}\circ\gamma)
 \le\frac1{\sqrt3}\operatorname{length}_{\HH}(\gamma).
\end{equation}
Indeed the hyperbolic length element is
$\sqrt{|\dd x|^2+\dd t^2}/t$ and $t\le1/\sqrt3$ on every face.
The inequality persists across their common edges. By
\eqref{eq:source-scales-projection}, a path from the chosen incenter
to the boundary of a patch whose projected boundary is at Euclidean
distance at least $D$ has length at least $\sqrt3D$.
The same bound holds in the complete source cover up to the first exit
from the patch.

Fix the intrinsic radius $R>0$ in the statement. The distance hypothesis allows us, for
all sufficiently large indices, to choose a common finite lattice
patch containing the chosen face, with every boundary point at
Euclidean distance greater than $(4R+1)/\sqrt3$ from its projected
incenter. We choose it inside the source interior. A path from the
incenter of length less than $4R+1$ cannot leave this patch, either on
$\Sig$ or on the complete source cover. Paths whose lengths approach
the distance from the incenter therefore determine the closed $R$-ball
on both surfaces.

To compare distances between two points of these balls, join the
incenter to the first by a path of length less than $R+\varepsilon$.
Their mutual distance is at most $2R$, by paths through the incenter.
Any competing path of length less than $2R+\varepsilon$, preceded by
the first path, has length less than $3R+2\varepsilon$ and hence stays
in the same patch when $\varepsilon$ is small. The two local metrics
are identical there. Taking infima proves equality of the complete
surface distances. Transporting this identification by the lifted
canonical map $I_n$ gives the asserted onto isometry from the closed
$R$-ball about $o$ in $\Sig$ to the closed $R$-ball about $I_n(o)$ in
$\Sigma_n$. The identification is injective on the patch, so every path
contained in either ball has its corresponding path in the other, with
the same length when rectifiable.
\end{proof}

\begin{theorem}\label{thm:ring}
For a marked equilateral disk let $z_i=g(i)$, let $X$ be its reflected
vertex set, and let $d_i=\dist(z_i,X\setminus\{z_i\})$.
Let $k_i$ be the number of original doubled faces at $i$, and set $b_i=3$
at a marked corner and $b_i=1$ elsewhere. If $P_i$ is the perimeter of
\[
 \operatorname{conv}\left(\{0\}\cup
 \left\{\frac{x-z_i}{|x-z_i|^2}:x\in X\setminus\{z_i\}\right\}\right),
\]
there are canonical positive scales $r_i=b_i k_i/P_i$ such that
\[
 \frac{b_i k_i}{2\pi}d_i\le r_i\le\frac{b_i k_i}{2}d_i,
 \qquad |z_i-z_j|\le\sqrt{r_i r_j}
\]
for every original edge $ij$. Also $r_i/r_j\le(b_i k_i/2)^2$ on each
oriented original edge. In particular scales on an original edge with both endpoints in the
original interior have ratio in $[1/9,9]$, and interior vertices satisfy
\[
 \sum_{i\in\operatorname{int}P}r_i^2\le\frac{9\sqrt3}{\pi}.
\]
Diameter here means Euclidean diameter in the chosen upper-half-space
target realization; its horosphere sections specify formal intrinsic
decorations. The diameter-one circumference is $k_i/r_i$ on the doubled
quotient and $b_i k_i/r_i=P_i$ on the reflection cover, hence $6/r_i$
at an original interior vertex. The marked intrinsic geodesic image of
an original ideal edge $ij$, even if it is no longer a target Delaunay
edge, has lambda length $\sqrt{r_i r_j}$ for these endpoint decorations.
\end{theorem}

The scales are determined by the target cusp geometry. The edge bound
applies to the marked intrinsic image of every original edge, whether
or not it is a target Delaunay edge.

\begin{proof}
We compute cusp circumferences from an inverted polygon, compare the
source and target decorations, and then obtain the edge and area bounds.
All target lifts are taken in the reflection cover of
Proposition~\ref{prop:reflection-cover}.

\emph{Horospheres fix the normalization.}
We use the signed-distance normalization in
\cite[\S3.1, equation~(11)]{gu2018} and compute its upper-half-space form.
For two distinct finite ideal endpoints $a_0,a_1$ and horospheres of
Euclidean diameters $D_0,D_1$, the signed distance along their ambient
ideal geodesic is
\[
 \log\frac{|a_0-a_1|^2}{D_0D_1}.
\]
To check this, translate $a_0$ to $0$ and apply the hyperbolic inversion
$(x,t)\mapsto(x,t)/(|x|^2+t^2)$. The first horosphere becomes the plane
$t=1/D_0$; the other becomes a sphere of diameter
$D_1/|a_1-a_0|^2$. The joining geodesic is vertical, and integration of
$\dd t/t$ gives the formula, with its sign also covering overlapping
horospheres. Hence its lambda length is
\begin{equation}\label{eq:source-scales-ambient-lambda}
 \lambda_{a_0a_1}=\frac{|a_0-a_1|}{\sqrt{D_0D_1}}.
\end{equation}

In an ideal triangle with vertices $i,j,k$, the horocyclic corner arc
at $i$ has length
\begin{equation}\label{eq:source-scales-corner-arc}
 \frac{\lambda_{jk}}{\lambda_{ij}\lambda_{ik}}.
\end{equation}
For a normalization check put the three vertices at $\infty,0,1$ in
the upper half-plane. If the first horocycle is $y=H$ and the other
two have diameters $D_j,D_k$, the corner arc has length $1/H$.
The three lambda lengths are
$\sqrt{H/D_j}$, $\sqrt{H/D_k}$, and $1/\sqrt{D_jD_k}$, which give
the stated ratio. The formula extends to formal decorations by moving
the horocycles along the cusp.

\emph{The inverted polygon measures the target cusp.}
Fix a vertex $i$. Local finiteness and periodicity give $d_i>0$.
Translate $z_i$ to $0$ and write $J(x)=x/|x|^2$ for planar inversion.
Let
\[
 C_i=\operatorname{conv}\bigl(\{0\}\cup
       \{J(x-z_i):x\in X\setminus\{z_i\}\}\bigr).
\]
All its defining points have norm at most $1/d_i$. Moreover $0$ lies
in its interior. To see this, use the corner lattice contained in $X$
to choose three faraway points in narrow cones with directions separated
by $2\pi/3$. Their inverse vectors still lie in those cones and their
triangle contains $0$ in its interior. Thus $C_i$ contains a disk about
$0$. All sufficiently small inverse vectors lie inside this disk, and
only finitely many remaining inverse vectors occur by local finiteness
of $X$. It follows that $C_i$ is a finite convex polygon.

A circle through $0$ with center $\zeta\ne0$ has equation
$|x|^2=2\zeta\mathbin{\cdot}x$. Its open disk is transformed by $J$ into
the half-plane $2\zeta\mathbin{\cdot}J(x)>1$. Thus the disk is empty of
$X-z_i$ exactly when
\[
 C_i\subset\{y:2\zeta\mathbin{\cdot}y\le1\}.
\]
The boundary line supports $C_i$ if, in addition, it meets $C_i$.

Since $0$ lies in the interior of $C_i$, the line containing any side
can be written as $2\zeta\mathbin{\cdot}y=1$, with $C_i$ in the indicated
closed half-plane. Both endpoints of the side are inverse vectors of
distinct points of $X-z_i$. The corresponding circle therefore has
empty open disk and passes through $0$ and at least two other points
of $X-z_i$. The convex hull of all translated points on this circle
is therefore a Delaunay cell incident to the root.
Conversely, an incident Delaunay cell has an empty circumdisk through
the root and at least two other vertices. Their distinct inverse vectors
lie on the associated line, so its intersection with $C_i$ is a side.
All inverse vectors on this side come from vertices on that circle.
This proves the correspondence between sides of $C_i$ and incident
Delaunay cells.

Triangulate each such cyclic cell from $z_i$. The inverse vectors of its
other vertices subdivide the corresponding side in order, so these
collinear subdivisions do not change its length. For two consecutive
nonroot vertices $a_0,a_1$ of the cell, measured from $z_i$, the identity
\[
 |J(a_0)-J(a_1)|=\frac{|a_0-a_1|}{|a_0|\,|a_1|}
\]
and \eqref{eq:source-scales-corner-arc} show that this side length is
the corner arc for the diameter-one horosphere at $z_i$.
Summing around the full lifted star gives the formal circumference
$P_i$.

For a sufficiently small horosphere of diameter $\varepsilon$, this
is an embedded cusp computation and the circumference is
$\varepsilon P_i$. The reflection cover has cusp degree $b_i$, by
the monodromy computation in Proposition~\ref{prop:reflection-cover}.
The quotient circumference is therefore $\varepsilon P_i/b_i$.
Rescaling the horocycle gives the diameter-one target decoration
\[
 w_i^*=\frac{P_i}{b_i}.
\]
This equation specifies an intrinsic decoration even when the
diameter-one horosphere is too large to give an embedded cusp curve.

The polygon $C_i$ contains $0$ and a vector of norm $1/d_i$, so its
diameter is at least $1/d_i$. The two boundary arcs between diameter
endpoints each have length at least the diameter. Thus $P_i\ge2/d_i$.
For the upper bound, the perimeter of a convex polygon is the integral
over directions in $[0,\pi]$ of its width. Indeed, the sum of the
absolute projections of its side vectors is twice the width, and the
integral of the absolute projection of a vector of length $l$ is $2l$.
Since $C_i\subset\overline B(0,1/d_i)$, every width is at most $2/d_i$.
We obtain
\begin{equation}\label{eq:source-scales-perimeter}
 \frac2{d_i}\le P_i\le\frac{2\pi}{d_i}.
\end{equation}

\emph{Changing the decoration gives the original-edge lambda length.}
In the original equilateral double, the diameter-one decoration has
lambda length $h$ on every original edge, by
\eqref{eq:source-scales-ambient-lambda}. Each corner arc consequently
has length $1/h$, by \eqref{eq:source-scales-corner-arc}, and its cusp
circumference is $w_i^0=k_i/h$. Theorem~\ref{thm:finite-uniformization}
identifies this and the target decoration on the same marked
hyperbolic surface. Define
\[
 r_i=\frac{k_i}{w_i^*}=\frac{b_i k_i}{P_i}.
\]
These numbers are positive and intrinsic to the stated normalization.
The diameter-one circumferences on the doubled quotient and the cover
are therefore $k_i/r_i$ and $b_i k_i/r_i=P_i$, respectively.
Equation~\eqref{eq:source-scales-perimeter} gives both comparisons
with $d_i$.

Decreasing a horocycle circumference from $w_i^0$ to $w_i^*$ moves it
inwards by signed distance $\log(w_i^0/w_i^*)$. Thus the lambda length
of any fixed marked ideal arc is multiplied by
$\sqrt{w_i^0w_j^0/(w_i^*w_j^*)}$. This follows from the signed-distance
definition and does not require the arc to remain Delaunay. Since
$w_i^0/w_i^*=r_i/h$, every original edge $ij$ has target intrinsic
lambda length
\[
 \lambda_{ij}^*
 =h\sqrt{\frac{r_i}{h}\frac{r_j}{h}}
 =\sqrt{r_i r_j}.
\]
At an original interior vertex, $k_i=6$ and $b_i=1$, so the
diameter-one circumference is $w_i^*=6/r_i$, as asserted.

\emph{Intrinsic edge length bounds its target chord.}
Carry an original ideal edge $ij$ by the canonical intrinsic isometry.
Its lift is a complete intrinsic geodesic on the dome with endpoints
$z_i,z_j$ in the chosen closed triangle. This choice of endpoints
follows from the lift on the source half in
Proposition~\ref{prop:reflection-cover}. For a boundary edge the same
assertion follows on the fixed seam; no marked corner lies in the
interior of that edge.

Choose diameter-$\varepsilon$ horospheres at these endpoints, with
$\varepsilon>0$ small enough that they are disjoint and meet this
geodesic only in its terminal cusp regions. They represent the target
decorations multiplied by $\varepsilon$. The length of the intervening
intrinsic geodesic segment is therefore
\[
 2\log\frac{\lambda_{ij}^*}{\varepsilon}.
\]
Its length is at least the ambient hyperbolic distance between the two
horospheres, which, by the signed-distance computation above, is
$2\log(|z_i-z_j|/\varepsilon)$. Thus
\[
 |z_i-z_j|\le\lambda_{ij}^*=\sqrt{r_i r_j}.
\]
Since $z_j$ is a distinct member of $X$, we also have
$d_i\le|z_i-z_j|$. Combining this with the upper scale bound gives
\[
 r_i\le\frac{b_i k_i}{2}d_i
 \le\frac{b_i k_i}{2}|z_i-z_j|
 \le\frac{b_i k_i}{2}\sqrt{r_i r_j}.
\]
Division by $\sqrt{r_i r_j}$ and squaring prove the oriented ratio
bound. Applying it to an original edge with both endpoints in the source
interior, where
$b_i k_i=b_j k_j=6$, gives the interval $[1/9,9]$.

\emph{Disjoint disks give the area bound.}
For an original interior vertex, $z_i\in\operatorname{int}\Tri$.
Reflecting it in a closest side gives a distinct point of $X$ at
distance $2\dist(z_i,\partial\Tri)$. Hence
$d_i\le2\dist(z_i,\partial\Tri)$ and the disk $B(z_i,d_i/2)$ lies
inside $\Tri$. For distinct interior vertices $i,j$, both $d_i$ and
$d_j$ are at most $|z_i-z_j|$. These open disks are therefore pairwise
disjoint. As $r_i\le3d_i$ at interior vertices, we conclude that
\[
 \sum_{i\in\operatorname{int}P}r_i^2
 \le9\sum_{i\in\operatorname{int}P}d_i^2
 \le\frac{36}{\pi}\Area(\Tri)
 =\frac{9\sqrt3}{\pi}.
\]
The circumference calculation fixes the scales, the geodesic comparison
controls original edges and their ratios, and these disjoint disks
control their total squared scale.
\end{proof}

\section{Compactness on the closed triangle}\label{sec:04-compactness}

We keep the canonical maps and the source correction of
Section~\ref{sec:03-source-scales}. We prove compactness in both directions
on the closed triangle. The two inputs from the literature concern ordinary
Riemann maps and the Euclidean area of convex hyperbolic surfaces.

\begin{theorem}[Normalized Jordan-domain convergence]\label{thm:rado}
Let $\Omega$ be a bounded Jordan domain with three distinct compatible
boundary marks. Let bounded Jordan domains $\Omega_n$ have boundaries
converging uniformly to $\partial\Omega$ and compatible marks converging
to those of $\Omega$. Their normalized
ordinary Riemann maps $\phi_n:\Tri\to\overline{\Omega_n}$ converge uniformly
on $\Tri$ to the normalized Riemann map of the limiting marked domain.
\end{theorem}

The $K=1$ case of \cite[Theorem~6.3, p.~979]{luo2022} gives normalized
subsequential compactness. A fixed conformal map from the disk to the
triangle and uniqueness in Theorem~\ref{thm:jordan-riemann} give the
full-sequence statement above. We apply this result only to the ordinary
Riemann maps.

\begin{theorem}[Convex-surface area]\label{thm:convex-area}
The boundary of any hyperbolically convex set of dimension at least two
in the unit Poincare ball has Euclidean induced-path area at most $16\pi$.
A two-dimensional convex set is interpreted as its metric double.
\end{theorem}

This is \cite[Theorem~4.1]{luo2024}, with the numbering of
arXiv:1910.08001v5. Its area is measured in the Euclidean induced path
metric on the surface. We will apply it to portions of the boundary of
the full reflected convex hull.

In the following statement, $\phi$ denotes the normalized ordinary Riemann
map of the limiting marked domain. For each finite disk $P_n$, let
$\kappa$ be the correction supplied by Proposition~\ref{prop:source-correction}.

\begin{theorem}\label{thm:closed-compactness}
Let marked equilateral disks $P_n$ have mesh $h_n\to0$, uniformly convergent
Jordan boundaries tending to $\partial\Omega$, and marks tending to three
distinct compatible marks of the bounded Jordan domain $\Omega$.
Let $f_n:\Tri\to P_n$ be the canonical maps, $g_n=f_n^{-1}$, and $\phi_n$
the normalized ordinary Riemann maps. Every subsequence has a further
subsequence and a marked homeomorphism $F:\Tri\to\overline\Omega$ such that
\[
 f_n\longrightarrow F\quad\hbox{uniformly on }\Tri,
 \qquad g_n\phi_n\longrightarrow F^{-1}\phi
 \quad\hbox{uniformly on }\Tri.
\]
Along that further subsequence, the actual inverses converge to $F^{-1}$
uniformly on each compact subset of $\Omega$. For the entire original sequence, the canonical target vertices have
covering radius tending to zero, and the height of the reflected target dome tends uniformly to zero.
For each finite disk the corrected map satisfies
$\int_P|D(g\kappa^{-1})|^2\le200\pi$.
\end{theorem}

The conclusion gives a homeomorphic limit together with convergence of its
inverse. The estimates below allow the target Delaunay triangulation to
vary throughout the sequence.

\begin{proof}
We first bound the area of the reflected target surface and the energy of
the corrected inverse. Reflection gives a continuity estimate up to the
boundary and hence target vertex coverage. A local area estimate then
gives continuity in the other direction. Finally, we use the ordinary
Riemann charts to remove the correction and pass to mutually inverse
limits.

\smallskip\noindent\emph{The periodic target surface and its area.}
Fix one marked equilateral disk $P$, of mesh $h$, and put $V=V(P)$.
Write $f:\Tri\to P$ and $g=f^{-1}$ for its canonical maps, and put
$W=g(V)$. Let $\mathcal R$ be the group generated by reflection in the
three side lines of $\Tri$. Its triangular tiles cover the plane. The
orientation-preserving subgroup has the doubled triangle as its quotient;
the three cone points are the images of tile vertices. We use the
reflection cover of Proposition~\ref{prop:reflection-cover}:
\[
 X=\bigcup_{\sigma\in\mathcal R}\sigma(W),
 \qquad C=\CH(X\cup\{\infty\}),
 \qquad S=\partial C.
\]
Thus $X$ is locally finite and periodic. It contains all tile vertices, so
it is a $1$-net of the plane: every point is within distance $1$ of $X$.
The target half of the associated hyperbolic surface is the portion of
$S$ lying over $\Tri$. The lifted open source half has peripheral degree
one at its original interior vertices. At the three marked corners, the
full reflection cover has degree three over the doubled quotient, as in
Proposition~\ref{prop:reflection-cover}.

We describe the vertical graph used in this realization. Above a
Delaunay polygon with empty circumdisk $B(z_0,R_0)$, the surface lies on
\[
 t^2+|x-z_0|^2=R_0^2.
\]
Indeed, the affine function interpolating $|x|^2$ at the vertices of this
polygon is
\[
 L(x)=2\operatorname{Re}(x\overline{z_0})+R_0^2-|z_0|^2.
\]
The empty-disk condition says that every other lifted vertex
$(w,|w|^2)$ lies on or above this plane. These supporting faces form the
lower convex envelope of the lifted vertices, and
$t^2=L(x)-|x|^2$ on each face. The polygons tile the plane and the formulas
agree on common edges. Subdivision of a cocircular polygon leaves this
surface unchanged. Since every empty disk has radius at most $1$, the
height lies between $0$ and $1$.

Let $\overline S$ be $S$ with each ideal vertex in $X$ filled, retaining
the intrinsic conformal structure at those points. Vertical projection
extends to a homeomorphism
\[
 p:\overline S\longrightarrow\CC.
\]
Put $S_\Tri=p^{-1}(\Tri)$. Reflections preserve the graph, and their fixed
curves are the graph portions over the corresponding side lines.
Proposition~\ref{prop:reflection-cover} identifies the canonical target
half with $S_\Tri$, including its projection and its boundary marking.
In particular, all area estimates below concern the boundary of the
complete periodic convex hull $C$.

The upper-half-space to ball isometry
\[
 \mathcal B(x,t)=
 \frac{(2x,\,|x|^2+t^2-1)}{|x|^2+(t+1)^2}
\]
has Euclidean conformal scale
\[
 \lambda(x,t)=\frac{2}{|x|^2+(t+1)^2}.
\]
On $S_\Tri$ we have $|x|\le1$ and $t\le1$, so $\lambda\ge2/5$.
The hull $C$ is three-dimensional, since its ideal set contains a planar
lattice and infinity, and hence is not contained in a generalized circle.
Theorem~\ref{thm:convex-area}, applied to $\mathcal B(C)$, gives
\begin{equation}\label{eq:compactness-area}
 \Area(S_\Tri)
 \le\frac{25}{4}\Area\bigl(\partial\mathcal B(C)\bigr)
 \le100\pi.
\end{equation}
Here and below, surface area means Euclidean induced-path area. Filled
ideal vertices have area zero.

The conformal metric comparison needed here follows directly on these
locally finite ideal polygon surfaces. A face meeting a compact planar
set has all its vertices within distance $2$ of that set, since its
circumradius is at most $1$. Only finitely many faces therefore meet the
set. Near an edge interior, develop the two adjacent faces across their
common geodesic into the hyperbolic plane. In an intrinsic conformal
coordinate $\zeta$, write its metric as
$\mathrm ds_H=\vartheta(\zeta)|\mathrm d\zeta|$. On either developed
face the Euclidean induced metric is
\[
 \mathrm ds_E=t(\zeta)\vartheta(\zeta)|\mathrm d\zeta|.
\]
The height has matching traces along the edge, so this formula gives a
conformal length density across it. Its square is the surface area
density; the edges have area zero. These densities consequently compute
Euclidean induced lengths and areas in the intrinsic conformal charts.
The filled vertices are handled by the point-removal arguments below.

\smallskip\noindent\emph{Energy of the corrected inverse.}
Use the source half-surface $H_P$ and its projection $p_P$ from the
preliminaries. Write $\overline H_P$ for its cusp filling, including the
boundary vertices, and retain $p_P$ for the extension to $P$.
Proposition~\ref{prop:source-correction} supplies a conformal homeomorphism
$Q_P:\overline H_P\to P$ with $\kappa=Q_Pp_P^{-1}$.
Let $\widetilde I:\overline H_P\to S_\Tri$ be the extension of the lift
of the canonical intrinsic map $I$. It is conformal on the open source
half after filling its interior cusps and extends homeomorphically to the
closed disks. The projections satisfy $g=p\widetilde I p_P^{-1}$. Hence
\[
 Y=\widetilde I Q_P^{-1}:P\longrightarrow S_\Tri,
 \qquad u=g\kappa^{-1}=pY.
\]
The preceding metric comparison makes $Y$ a conformal parametrization for
the Euclidean induced metric away from the filled vertices. Almost
everywhere on the faces,
\[
 |Y_x|=|Y_y|,\qquad
 Y_x\cdot Y_y=0,\qquad
 |DY|^2=2|Y_x\mathbin{\times}Y_y|.
\]
The face edges have area zero and the traces agree across them. The area
formula and exhaustion away from the filled vertices therefore give
\[
 \int_P|DY|^2=2\Area(S_\Tri)\le200\pi.
\]
This is also a Sobolev energy bound across each filled interior vertex.
For example, in a local coordinate centered at such a point, a cutoff
which rises logarithmically from $0$ at radius $\epsilon^2$ to $1$ at
radius $\epsilon$ has squared gradient integral
$2\pi/\log(1/\epsilon)$. Inserting these cutoffs in the weak derivative
identity and letting $\epsilon\to0$ extends the bounded coordinate
functions of $Y$ across the puncture. The weak derivatives retain their
$L^2$ bounds. Since vertical projection does not increase Euclidean
length, we obtain
\begin{equation}\label{eq:compactness-energy}
 u\in W^{1,2}(\operatorname{int}P,\RR^2),
 \qquad \int_P|Du|^2\le200\pi.
\end{equation}

\smallskip\noindent\emph{The source annulus estimate.}
We use Theorem~\ref{thm:annulus-duality}: a plane annulus of modulus $M$
has crossing-curve extremal length $M$ and essential-loop extremal length
$1/M$. These identities are preserved by conformal equivalence.
Fix $z\in\operatorname{int}P$ and radii satisfying
\[
 h<r<R,\qquad R-r>2h,\qquad
 \overline B(z,R+h)\subset\operatorname{int}P.
\]
For $A=\{r<|x-z|<R\}$, Proposition~\ref{prop:source-correction} gives
\[
 M:=\Mod(p_P^{-1}(A))
 \ge\frac{1}{2\pi}\log\frac{R-h}{r+h}.
\]
All interior cusp points in this annulus are filled. The map $Q_P$
identifies it conformally with the plane annulus $\kappa(A)$, whose two
boundary components are Jordan curves. Its image under $\widetilde I$
has the same modulus. Equip that image with the Euclidean induced metric
of $S_\Tri$, which is a conformal metric of area at most $100\pi$ by
\eqref{eq:compactness-area}.

Set $D=\diam g(\overline B(z,r))$. Every essential Jordan loop in the
image annulus projects to a Jordan curve enclosing
$g(\overline B(z,r))$. Its Euclidean length is at least $2D$: projection
onto the line joining any two points of the enclosed set traverses their
separation in both directions. Its length on the surface is at least
that projected length. Hence the definition of extremal length gives
\[
 \frac1M\ge\frac{(2D)^2}{100\pi}.
\]
Filled points cause no change in this argument. Curves reaching one point
have analytic $2$-modulus zero, as the density
$1/(|\zeta|\log(1/|\zeta|))$ in a sufficiently small conformal coordinate
disk has finite squared integral and infinite integral along every curve
reaching its center. One may therefore restrict the curve families to
those avoiding the finitely many filled points. We conclude that
\begin{equation}\label{eq:compactness-source-annulus}
 \diam g(\overline B(z,r))^2
 \le\frac{50\pi^2}{\log((R-h)/(r+h))}.
\end{equation}
In particular, whenever compact subsets of $\Omega$ are eventually
contained in $\operatorname{int}P_n$, this estimate gives asymptotic
equicontinuity of $g_n$ on those compact sets. For points $x_n,y_n$ in a
fixed compact set with $|x_n-y_n|\to0$, choose a fixed admissible outer
radius and take $r_n=\max\{2h_n,|x_n-y_n|\}$. The denominator in
\eqref{eq:compactness-source-annulus} tends to infinity.

\smallskip\noindent\emph{Continuity on the closed reference triangle.}
Let $\varphi:\Tri\to P$ be the normalized ordinary Riemann map supplied
by Theorem~\ref{thm:jordan-riemann} for the fixed finite disk, and set
\[
 v=u\varphi=g\kappa^{-1}\varphi.
\]
This homeomorphism fixes the three corners and maps each side onto itself.
Conformal invariance of Dirichlet energy and
\eqref{eq:compactness-energy} give
\[
 \int_\Tri|Dv|^2\le200\pi.
\]
On the tile $\sigma(\Tri)$, $\sigma\in\mathcal R$, extend $v$ by
$\widetilde v=\sigma v\sigma^{-1}$. The definitions agree on shared sides
because $v$ preserves each side, and they agree at tile vertices because
$v$ fixes the corners. Reflecting the inverse gives the inverse of this
extension, so $\widetilde v$ is a homeomorphism of the plane. Sobolev
gluing along sides, followed by the same point-removal argument at tile
vertices, gives $\widetilde v\in W^{1,2}_{\mathrm{loc}}(\CC)$.

Every unit tile meeting $B(x,1)$ is contained in $B(x,2)$. The tiles have
disjoint interiors and area $\sqrt3/4$, so at most $30$ meet this disk.
Thus
\[
 \int_{B(x,1)}|D\widetilde v|^2\le30\cdot200\pi.
\]
For $0<s<1$, put
$D_s=\diam\widetilde v(\overline B(x,s))$. For almost every
$t\in(s,1)$ the image of the circle of radius $t$ is a rectifiable Jordan
curve enclosing this set. Its length $\mathcal L(t)$ satisfies
$\mathcal L(t)\ge2D_s$.
Cauchy--Schwarz, with $\nu$ the unit tangent to the circle, gives
\[
 \mathcal L(t)^2\le2\pi t
 \int_{|\zeta-x|=t}|D\widetilde v(\zeta)\nu|^2\dd s.
\]
Integration against $\dd t/t$ yields
\[
 4D_s^2\log(1/s)
 \le2\pi\int_{B(x,1)}|D\widetilde v|^2
 \le12000\pi^2.
\]
Define a modulus of continuity $\mu$ by
\[
 \mu(0)=0,\qquad
 \mu(s)=\min\left\{1,\sqrt{\frac{3000\pi^2}{\log(1/s)}}\right\}
 \quad(0<s<1),\qquad
 \mu(s)=1\quad(s\ge1).
\]
Since $\diam\Tri=1$, we have the uniform estimate
\begin{equation}\label{eq:compactness-forward-modulus}
 |v(x)-v(y)|\le\mu(|x-y|),\qquad x,y\in\Tri.
\end{equation}
It includes points on the sides and at the corners.

\smallskip\noindent\emph{Small source faces give target coverage.}
Define
\[
 m=\max_{\text{closed source faces }\tau}
       \diam\bigl(\varphi^{-1}(\tau)\bigr),
 \qquad \eta=\varphi^{-1}\kappa\varphi,
 \qquad G=g\varphi.
\]
The points $\varphi(x)$ and $\kappa\varphi(x)$ lie in a common closed
source face. Therefore $\|\eta-\id\|_\infty\le m$ and, since $G=v\eta$,
\begin{equation}\label{eq:compactness-correction}
 \|G-v\|_\infty\le\mu(m),\qquad
 \|\eta^{-1}-\id\|_\infty\le m.
\end{equation}
For the second inequality substitute $y=\eta(x)$ in the displacement
bound for $\eta$.

Let
\[
 \delta=\sup_{w\in\Tri}\dist(w,W)
\]
be the covering radius of the target vertices. Given $w\in\Tri$, choose
$x\in\Tri$ with $v(x)=w$, and choose a vertex $a_\tau$ of a closed source
face containing $\varphi(x)$. Put $x_\tau=\varphi^{-1}(a_\tau)$. Then
$|x-x_\tau|\le m$ and $v(x_\tau)=g(a_\tau)$, because $\kappa$ fixes
vertices. By \eqref{eq:compactness-forward-modulus},
\begin{equation}\label{eq:compactness-coverage}
 \delta\le\mu(m).
\end{equation}
Reflection makes $X$ a $\delta$-net of the entire plane. Every empty
circumdisk consequently has radius at most $\delta$, and the graph
formula gives
\begin{equation}\label{eq:compactness-height}
 0\le t\le\delta\quad\hbox{on }\overline S.
\end{equation}

\smallskip\noindent\emph{Local area controls target annuli.}
We now obtain an estimate for $v^{-1}$. For $z\in\CC$ and $s\ge\delta$,
apply the hyperbolic isometry
\[
 (x,t)\longmapsto\bigl((x-z)/s,t/s\bigr).
\]
By \eqref{eq:compactness-height}, the transformed graph over $B(z,s)$
lies in $|x|<1$, $0\le t\le1$. The conformal scale of $\mathcal B$ is
again at least $2/5$ there. Apply Theorem~\ref{thm:convex-area} to the
transformed whole hull and undo the dilation. We obtain
\begin{equation}\label{eq:compactness-local-area}
 \Area\bigl(p^{-1}(B(z,s))\bigr)\le100\pi s^2,
 \qquad s\ge\delta.
\end{equation}

Fix $\delta\le r<R$, write $\ell=\log(R/r)$, and let
\[
 A_H=p^{-1}\bigl(\{r<|x-z|<R\}\bigr)
\]
be the filled conformal annulus. On $A_H$ multiply the Euclidean induced
length metric by the density $\rho(x,t)=1/|x-z|$. If
$\mathcal A_z(s)=\Area(p^{-1}(B(z,s)))$, Stieltjes integration and
\eqref{eq:compactness-local-area} give
\begin{align*}
 \int_{A_H}\rho^2\dd A
 &\le\frac{\mathcal A_z(R)}{R^2}
      +2\int_r^R\frac{\mathcal A_z(s)}{s^3}\dd s\\
 &\le100\pi(1+2\ell).
\end{align*}
The omitted inner boundary term is nonpositive. The area distribution
$\mathcal A_z$ is continuous: a graph portion over a circle has surface area zero,
as follows face by face. Thus the displayed integration is valid also
when a boundary circle contains vertices.

Along every crossing curve, Euclidean surface arclength dominates the
absolute variation of the horizontal radius. Its $\rho$-length is
therefore at least $\ell$. This also holds across a filled vertex by
continuity and partitioning the curve. The crossing-family definition of
extremal length now gives
\begin{equation}\label{eq:compactness-target-annulus}
 \Mod(A_H)\ge\frac{\ell^2}{100\pi(1+2\ell)}
 \ge\frac{\ell}{300\pi}\quad\hbox{if }\ell\ge1.
\end{equation}

\smallskip\noindent\emph{Continuity of the inverse in the reflected chart.}
The conformal homeomorphism
\[
 q=\varphi^{-1}Q_P\widetilde I^{-1}:S_\Tri\longrightarrow\Tri
\]
respects the sides and corners. On the portion over $\sigma(\Tri)$,
extend it by $\sigma q\sigma^{-1}$, using the ambient reflection on the
surface and the planar reflection on its image. The maps agree along
shared sides. The ambient reflection restricts to an anticonformal
involution of the intrinsic surface, and its fixed curve is the side
curve. Theorem~\ref{thm:reflection} therefore makes the extension
conformal across that curve away from marked points. At a marked point
the glued map is continuous and bounded in a local cusp-filling
coordinate, so Theorem~\ref{thm:removable} applies. Six reflected sectors meet at a
tile vertex; the same argument applies there. The extension and the
tilewise inverses are continuous. Thus we obtain a conformal
homeomorphism
\[
 Q:\overline S\longrightarrow\CC.
\]
It sends the graph over every closed tile onto that same tile.

Put $\Psi=Qp^{-1}$. This planar homeomorphism preserves every tile and
satisfies
\[
 |\Psi(w)-w|\le1,\qquad
 \Psi|_\Tri=\varphi^{-1}\kappa f=v^{-1}.
\]
For a planar annulus $A$ we have
\begin{equation}\label{eq:compactness-chart-modulus}
 \Mod(\Psi(A))=\Mod(p^{-1}(A)),
\end{equation}
because $Q$ is conformal.

Take $\delta\le r\le e^{-1}$ and
$A=\{r<|w-z|<1\}$. Tile preservation gives
$\Psi(B(z,1))\subset B(z,2)$, hence $\Area(\Psi(A))\le4\pi$.
Every essential Jordan loop in $\Psi(A)$ encloses
$\Psi(\overline B(z,r))$. Writing
$d_r=\diam\Psi(\overline B(z,r))$, the same line-projection argument
used for \eqref{eq:compactness-source-annulus} bounds its length below
by $2d_r$. Applying Theorem~\ref{thm:annulus-duality} to this plane
annulus with Jordan boundary components, the Euclidean conformal metric
on $\Psi(A)$ yields
\[
 \frac1{\Mod(\Psi(A))}\ge\frac{(2d_r)^2}{4\pi}
 =\frac{d_r^2}{\pi}.
\]
Together with \eqref{eq:compactness-target-annulus} and
\eqref{eq:compactness-chart-modulus}, this proves
\[
 d_r^2\le\frac{300\pi^2}{\log(1/r)}.
\]
For $x,y\in\Tri$ with $|x-y|\le r$, center the annulus at $x$ and use
$\Psi|_\Tri=v^{-1}$. Since the values lie in $\Tri$, we conclude that
\begin{equation}\label{eq:compactness-inverse-modulus}
 |v^{-1}(x)-v^{-1}(y)|^2
 \le\min\left\{1,\frac{300\pi^2}{\log(1/r)}\right\},
 \qquad \delta\le r\le e^{-1}.
\end{equation}
The reflected construction makes this estimate valid at boundary points
as well as interior points.

\smallskip\noindent\emph{Ordinary charts and removal of the correction.}
Return to the sequence, using subscripts for all the preceding
constructions and taking $\varphi=\phi_n$ for the $n$th disk.
Theorem~\ref{thm:rado} gives
\[
 \phi_n\longrightarrow\phi\quad\hbox{uniformly on }\Tri.
\]
It follows that $m_n\to0$. To verify this, suppose that two points
$x_n,y_n$ in a common closed source face have preimages whose distance is
bounded below along a subsequence. A further subsequence has
\[
 \phi_n^{-1}(x_n)\longrightarrow\alpha,\qquad
 \phi_n^{-1}(y_n)\longrightarrow\beta,\qquad \alpha\ne\beta.
\]
Since $|x_n-y_n|\le h_n\to0$, uniform convergence of $\phi_n$ gives
$\phi(\alpha)=\phi(\beta)$, contradicting injectivity. Taking pairs which
realize the maximal face diameter proves the assertion. Equations
\eqref{eq:compactness-coverage} and \eqref{eq:compactness-height} now give
$\delta_n\to0$ and uniform convergence of the reflected dome heights to
zero for the entire original sequence.

The maps $v_n$ share the modulus \eqref{eq:compactness-forward-modulus}.
Their inverses are asymptotically equicontinuous by
\eqref{eq:compactness-inverse-modulus}: if $|x_n-y_n|\to0$, choose
\[
 r_n=\max\{\delta_n,|x_n-y_n|\}.
\]
The finite target sets have positive covering radius, so $r_n>0$, and
eventually $r_n\le e^{-1}$. The right side of
\eqref{eq:compactness-inverse-modulus} then tends to zero.
Arzel\`a--Ascoli, applied successively to these two families, gives a
further subsequence of any prescribed subsequence on which
\[
 v_n\longrightarrow v_\infty,\qquad
 v_n^{-1}\longrightarrow w_\infty
 \quad\hbox{uniformly on }\Tri.
\]
Passing to the identities $v_nv_n^{-1}=v_n^{-1}v_n=\id$ gives
$v_\infty w_\infty=w_\infty v_\infty=\id$. For example, the error in
$v_n(v_n^{-1}(x))\to v_\infty(w_\infty(x))$ is bounded by
$\|v_n-v_\infty\|_\infty$ plus the modulus of continuity of
$v_\infty$ evaluated at $\|v_n^{-1}-w_\infty\|_\infty$.
Thus $v_\infty$ and $w_\infty$ are mutually inverse homeomorphisms.

By \eqref{eq:compactness-correction},
$\|\eta_n^{\pm1}-\id\|_\infty\le m_n\to0$. The exact identities
\[
 g_n\phi_n=v_n\eta_n,\qquad
 \phi_n^{-1}f_n=\eta_n^{-1}v_n^{-1}
\]
therefore imply
\[
 g_n\phi_n\longrightarrow v_\infty,\qquad
 \phi_n^{-1}f_n\longrightarrow w_\infty
 \quad\hbox{uniformly on }\Tri.
\]
Set $F=\phi w_\infty$. Uniform convergence of the ordinary charts gives
$f_n\to F$ uniformly on $\Tri$. This $F$ is a homeomorphism onto
$\overline\Omega$, has the prescribed corner values, and satisfies
$F^{-1}\phi=v_\infty$.

For the selected subsequence, let $K\Subset\Omega$. Uniform convergence of the boundary
parametrizations makes $K\subset\operatorname{int}P_n$ eventually: the
straight homotopy between sufficiently close parametrizations avoids $K$
and preserves winding number about each of its points. We also have
\[
 \phi_n^{-1}\longrightarrow\phi^{-1}\quad\hbox{uniformly on }K.
\]
Indeed, otherwise choose $x_n\in K$ witnessing failure, pass to limits
$x_n\to x$ and $\phi_n^{-1}(x_n)\to\alpha$, and use uniform convergence
of $\phi_n$ to get $\phi(\alpha)=x$. This forces
$\alpha=\phi^{-1}(x)$, a contradiction. Consequently
\[
 g_n|_K=(g_n\phi_n)\phi_n^{-1}|_K
 \longrightarrow v_\infty\phi^{-1}|_K=F^{-1}|_K
\]
uniformly. The energy estimate proves the finite-disk assertion, target
coverage controls the dome heights, and the two continuity estimates
give the mutually inverse closed-domain limits.
\end{proof}

\section{Whole marked limits of lattice domes}\label{sec:05-lattice-limits}

We use the lattice and reflection-cover notation of
Sections~\ref{sec:02-preliminaries} and~\ref{sec:03-source-scales}.
We first record the conformal and metric inputs and establish the
compactification needed to identify the ends. We then prove that a
normalized lattice limit fills the entire convex boundary and that each
of its finite cusp labels remains isolated in the approximating ideal sets.

\begin{theorem}[Conformal metrics on convex boundaries]\label{thm:convex-conformal}
Let $Y$ be a compact proper ideal set whose hyperbolic convex hull has
dimension at least two, and let $S=\partial\CH(Y)$ in the unit Poincare
ball, with the doubled convention in dimension two. The induced Euclidean
and hyperbolic path metrics give the same conformal structure on each
boundary component. In its conformal charts the Euclidean path metric has
a conformal length density, with the corresponding squared density as area
measure. Its path lengths and induced area can therefore be used to
compute extremal lengths on this Riemann surface.
\end{theorem}

This is \cite[Theorem~11.2]{luo2024} applied to the convex boundary,
as in \S11.2; see also
\cite[\S2.3]{luo2024}. It permits conformal area arguments even when the
convex boundary is nonsmooth.

\begin{theorem}[Intrinsic-distance convergence]\label{thm:alexandrov}
Let compact convex bodies $K_n$ in Euclidean or hyperbolic three-space
converge in Hausdorff distance to a compact convex body $K$ with nonempty
interior, and suppose each $K_n$ has nonempty interior. Put
$S_n=\partial K_n$ and $S=\partial K$. If $x_n,y_n\in S_n$ converge to
$x,y\in S$, respectively, then their intrinsic distances converge:
$d_{S_n}(x_n,y_n)\to d_S(x,y)$.
\end{theorem}

This is the compact, full-dimensional case of
\cite[Theorem~5.1]{luo2024}. We will apply it to convex boundary spheres
and then exclude paths through the artificial caps.

\begin{lemma}\label{lem:lattice-limits-source}
The following hold for the lattice surface $\Sig$:
\begin{enumerate}[label=(\roman*)]
\item The surface is complete and proper in its intrinsic metric.
\item There is an orientation-preserving conformal homeomorphism
$q:\CC\setminus\Lam\to\Sig$ that carries each punctured Euclidean face
to its corresponding ideal face and extends across each $v\in\Lam$
to the cusp labeled $v$. After these cusps are filled, the extension is
a conformal homeomorphism from $\CC$ onto the filled surface.
\end{enumerate}
\end{lemma}

Thus the facewise conformal identification equips the filled lattice
surface with the ordinary plane coordinate.

\begin{proof}
We extend the finite source identification to the lattice and then check
that intrinsic balls are compact.

\emph{Conformal filling.}
Each empty circumcircle of a unit lattice face gives its supporting
hemisphere in the lattice hull. The lifted faces are therefore exactly
the ideal triangles used in the source construction. A conformal map
between an ideal face and its Euclidean equilateral face is uniquely
determined by the three ordered vertex values, by
Theorem~\ref{thm:jordan-riemann}
\cite[Theorems~2.22--2.24 and~4.5]{mcmullen2025}. Consequently
the face restrictions of the identifications in
Proposition~\ref{prop:source-correction} agree on overlapping finite
equilateral patches. Apply that proposition in a finite hexagonal patch
containing any prescribed face, edge, or complete vertex star. It gives
edge agreement and conformality across the edges and at the filled cusp.
These local identifications therefore glue over the whole lattice.
Their inverses give $q$. The triangulation is locally finite, and the maps
and their inverses preserve every closed face. They consequently give a
homeomorphism on the filled surfaces, conformal at every finite point.

\emph{Compact intrinsic balls.}
Let $p_{\mathrm{vert}}$ be vertical projection on the lattice dome. Every lifted face
has height at most $1/\sqrt3$. Hence every rectifiable curve $\gamma$
on $\Sig$ satisfies
\[
 \operatorname{length}_{E}(p_{\mathrm{vert}}\gamma)
 \le \frac1{\sqrt3}\operatorname{length}_{\Sig}(\gamma).
\]
A ball of bounded intrinsic radius projects into a bounded planar disk,
which meets only finitely many lattice faces. At each lattice vertex the
six ideal sectors form a complete cusp. In a horodisk coordinate the
distance from a fixed horocycle to height $s$ grows as $\log s$; thus the
ball avoids a sufficiently deep tail of each of these finitely many
cusps. The remaining portions of the finitely many closed faces form a
compact set. Closed intrinsic balls are therefore compact. An intrinsic
Cauchy sequence lies in one such ball and converges, proving completeness.
\end{proof}

\begin{lemma}\label{lem:lattice-limits-component}
Let $E$ be a proper compact subset of the ideal sphere of the unit
Poincar\'e ball. Let $C=\CH(E)$ have dimension at least two, let
$S=\partial C$ have its intrinsic metric and conformal structure, and
let $S_0$ be a connected component of $S$.
Let $A\subset\CC$ be closed and discrete, and suppose that
$j:\CC\setminus A\to S_0$ is a conformal homeomorphism.
Set
\[
 \Phi(\xi)=\frac{2\xi}{1+|\xi|^2},\qquad
 K=\overline{\Phi(C)}.
\]
Let $\overline S$ be $\partial K$ when $K$ is three-dimensional and
the topological double of $K$ along its entire planar boundary when
$K$ is two-dimensional. In the latter case regard $\Phi$ as a map
to the corresponding sheet of the double.
Then the following hold:
\begin{enumerate}[label=(\roman*)]
\item $S_0=S$.
\item The map $\Phi\circ j$ extends to a homeomorphism
$H:\Chat\to\overline S$ such that $H(A\cup\{\infty\})=E$,
every $H(a)$ with $a\in A$ is isolated in $E$, and $H(\infty)$ is
the only possible accumulation point of $E$.
\end{enumerate}
\end{lemma}

The hypothesis identifies only one boundary component. The conclusion
both fills its punctures and excludes every other component.

\begin{proof}
We put a compact path metric on the whole boundary sphere. The area bound
then supplies short circles around the punctures and near infinity.
Their complementary disks determine all the missing points.

\emph{A compact boundary sphere.}
In Klein coordinates $K$ is the Euclidean convex hull of $E$, and
$K\cap\mathbb S^2=E$. Indeed, if $\zeta\in K\cap\mathbb S^2$,
convex combinations of points of $E$ approach $\zeta$. Their scalar
products with $\zeta$ approach $1$, so points of $E$ have scalar products
with $\zeta$ arbitrarily close to $1$. Such points approach $\zeta$,
which belongs to the closed set $E$.

When $K$ has interior, equip $\partial K$ with its Euclidean intrinsic
path metric $d_K$. Choose $b$ and $0<r_*<R_*$ with
$\overline B(b,r_*)\subset K\subset\overline B(b,R_*)$.
The gauge $p$ of $K-b$ satisfies
\[
 |p(u)-p(w)|\le\frac{|u-w|}{r_*}.
\]
This follows from subadditivity and $p(u-w)\le|u-w|/r_*$, followed by
interchanging $u$ and $w$. On the unit sphere $p\ge1/R_*$, so the radial
function $1/p$ is Lipschitz. The radial parametrization of $\partial K$
is thus Lipschitz for lengths of spherical curves. It is a continuous
surjection into $(\partial K,d_K)$, making this metric space compact.
Since Euclidean distance is at most $d_K$, its topology is exactly the
boundary-sphere topology.

When $K$ is planar, equip its double with the doubled Euclidean path
metric. Within one copy the distance is Euclidean, and between opposite
copies it is
\[
 d_K(x,x')=
 \min_{c\in\partial_{\mathrm{rel}}K}
       \bigl(|x-c|+|c-x'|\bigr),
\]
where $\partial_{\mathrm{rel}}K$ denotes the boundary in its plane.
This metric induces the compact topology of the two disks glued along
their boundaries. Thus in both dimensions $\overline S$ is a compact
metric sphere, and $\Phi$ identifies $S$ with $\overline S\setminus E$,
with the sheets retained in dimension two.

The tangential and radial singular values of $D\Phi$ are respectively
\[
 \frac2{1+|\xi|^2},\qquad
 \frac{2(1-|\xi|^2)}{(1+|\xi|^2)^2}.
\]
Both are at most $2$. Klein path lengths are therefore at most twice
the corresponding induced Euclidean path lengths in the Poincar\'e
ball. The same statement holds on the double, sheet by sheet.

\emph{Short circles from finite conformal area.}
By Theorem~\ref{thm:convex-conformal}, pulling the induced Euclidean
path metric back by the conformal homeomorphism $j$ gives a length
density $\rho(z)|\mathrm dz|$ and its corresponding squared area
density. By
Theorem~\ref{thm:convex-area}, equivalently
\cite[Theorem~4.1]{luo2024},
\[
 \int_{\CC\setminus A}\rho^2\dd x\dd y
   =\Area(S_0)\le\Area(S)\le16\pi.
\]
Here the areas are for the induced Euclidean path metrics, with both
sheets counted in dimension two. Injectivity of $j$ ensures that the
pullback has no additional multiplicity.

For almost every $t>0$, let $L_a(t)$ be the Euclidean length of the
$j$-image of $|z-a|=t$. We discard radii for which the circle meets $A$.
Cauchy--Schwarz on each circle and polar integration give
\begin{equation}\label{eq:lattice-limits-circles}
 \int_r^R L_a(t)^2\frac{\mathrm dt}{t}
 \le 2\pi\int_{\{r<|z-a|<R\}\setminus A}\rho^2\dd x\dd y
 \le32\pi^2.
\end{equation}
The measure $\mathrm dt/t$ has infinite mass at both $0$ and infinity.
For each $a\in A$ there are consequently radii $t_k\downarrow0$ with
$L_a(t_k)\to0$. There are also radii $R_k\uparrow\infty$ with
$L_0(R_k)\to0$. We choose strictly monotone sequences of rectifiable
image circles, avoiding the countably many radii that meet $A$.

\emph{The small side of a short curve.}
Fix a homeomorphism from $(\overline S,d_K)$ to the round sphere.
Both it and its inverse are uniformly continuous. A Jordan curve of
$d_K$-length tending to zero maps into a spherical ball whose radius
tends to zero. The complement of that ball is connected, so one of the
two closed Jordan disks bounded by the curve lies in the ball. Its
$d_K$-diameter tends to zero. If two fixed distinct reference points
lie on the same side of all the curves, this small disk is eventually
the side opposite the reference points: a disk of vanishing diameter
cannot contain both of them.

\emph{Filling each finite puncture.}
Fix $a\in A$, and choose $\epsilon>0$ such that
$\overline B(a,\epsilon)\cap A=\{a\}$.
Choose two distinct points of
$\CC\setminus(A\cup\overline B(a,\epsilon))$ as exterior reference
points. Their images under $\Phi j$ lie on the same side of every
image circle of radius less than $\epsilon$, because this exterior
source region is connected. Local openness across each circle puts
the image of the punctured interior disk on the opposite side.

Take the radii $t_k\downarrow0$ furnished by
\eqref{eq:lattice-limits-circles}, all less than $\epsilon$, and let
$D_k$ be the closed Jordan disk opposite the reference points. These
disks are nested. Their boundary lengths in $d_K$ are at most
$2L_a(t_k)$, so their diameters tend to zero. Their intersection is
one point, denoted $\eta_a$. The map $\Phi j$ extends continuously
at $a$ with value $\eta_a$.

The point $\eta_a$ belongs to $E$. If it were in another component of
$\overline S\setminus E$, an open neighborhood in that component would
be disjoint from the converging images in $\Phi(S_0)$. If it were in
$\Phi(S_0)$, continuity of the inverse would make points tending to
$a$ converge to a point of $\CC\setminus A$. Both conclusions are
impossible.

To see that $\eta_a$ is isolated in the entire set $E$, fix one radius
$t_0<\epsilon$ larger than all subsequent $t_k$ and define $D_0$ in
the same way. The closed source annulus
$t_k\le|z-a|\le t_0$ maps onto the entire closed annular region between
$\partial D_k$ and $\partial D_0$. Its image lies in this region,
is open in the region away from the boundary curves by local openness,
and is relatively closed there by compactness. The open annular region
is connected, so the image fills it. Taking the union over $k$ gives
\[
 D_0\setminus\{\eta_a\}\subset\Phi(S_0),\qquad
 D_0\cap E=\{\eta_a\}.
\]
The extension maps the closed source disk homeomorphically onto $D_0$.
Distinct punctures have distinct images: otherwise two disjoint small
punctured source disks would map onto punctured neighborhoods of the
same point, contrary to injectivity. Filling every point of $A$ gives
a homeomorphic open embedding
\[
 j_1:\CC\longrightarrow\overline S,
 \qquad j_1(\CC)=\Phi(S_0)\cup\{\eta_a:a\in A\}.
\]
Closed discreteness of $A$ makes the preceding local extensions
compatible at every finite source point.

\emph{The remaining point and the whole boundary.}
Use the radii $R_k\uparrow\infty$ from
\eqref{eq:lattice-limits-circles}. The images under $j_1$ of the closed
disks $|z|\le R_k$ are closed Jordan disks. Their complementary closed
Jordan disks $E_k$ are nested, and
\[
 \bigcap_k E_k=\overline S\setminus j_1(\CC).
\]
Two fixed distinct finite source points lie inside the increasing
source disks for all large $k$. Their images select $E_k$ as the small
side of its boundary, whose $d_K$-length is at most $2L_0(R_k)$.
Thus $\diam_{d_K}(E_k)\to0$, and the displayed intersection is a
single point $\eta_\infty$.

Any other component of $S$ would give a nonempty open subset of
$\overline S$ disjoint from $j_1(\CC)$ and hence contained in
$\{\eta_\infty\}$. This proves $S_0=S$. It also proves
$\eta_\infty\in E$, since $j_1(\CC)$ now contains
$\overline S\setminus E$. Setting $H(\infty)=\eta_\infty$ extends
$j_1$ continuously to $\Chat$, by the shrinking of $E_k$.
The extension is a continuous bijection between compact Hausdorff
spheres, hence a homeomorphism. It sends $A\cup\{\infty\}$ exactly
onto $E$. Since a closed discrete subset of $\CC$ has no finite
accumulation point, its image has only the possible accumulation point
$H(\infty)$, as asserted.
\end{proof}

\begin{theorem}\label{thm:marked-limit}
Let marked equilateral disks contain original faces whose distances to
the source boundary, divided by the mesh, tend to infinity. Identify these
faces with a fixed lattice face by direct source similarities, choose its
incenter $o\in\Sig$, and let $I_n$ be the consistently lifted canonical
intrinsic maps into the reflected target domes. Write $x_n(v)$ for the
marked ideal image of a fixed lattice label when defined,
$I_n(o)=(a_n,t_n)$, $M_n(z,t)=((z-a_n)/t_n,t/t_n)$ and
$Y_n=M_n(X_n\cup\{\infty\})$.
Every subsequence has a further subsequence for which $Y_n$ converges in
spherical Hausdorff distance to a proper compact set $Y$, the hull
$\CH(Y)$ is three-dimensional, and $M_nI_n$ converges on intrinsic compact
sets to an isometry $J:\Sig\to\partial\CH(Y)$ onto the whole boundary.
Let $K_Y$ be the closed Klein convex hull after a fixed change of model.
The isometry, composed with that change, extends to a homeomorphism
$\overline J:\widehat\Sig\to\partial K_Y$. On the added ideal labels it
induces a bijection $y:\Lam\cup\{\infty\}\to Y$ with
$y(\infty)=\infty$ and $M_nx_n(v)\to y(v)$ for every finite $v$.
For each such $v$ there is a fixed spherical neighborhood $U_v$ with
\[
 Y_n\cap U_v=\{M_nx_n(v)\}\quad\hbox{for every sufficiently large }n.
\]
The only accumulation point of $Y$ is infinity.
\end{theorem}

In particular, the conclusion controls the whole ideal set, including
reflected vertices whose source labels leave every fixed patch.

\begin{proof}
We first extract a convex limit and identify one of its boundary
components by convergence of full intrinsic distances. The conformal
compactification then exhausts the boundary and excludes a planar
limit. Finally, peripheral curves preserve each cusp label and isolate
it from every other approximating ideal point.

\emph{Exact metric balls and normalization.}
Start with any subsequence and put $J_n=M_n I_n$.
Each $M_n$ is an ambient hyperbolic isometry fixing ideal infinity, and
$J_n(o)=(0,1)$. Choose a fixed isometry $\Psi$ from upper half-space
to the unit Poincar\'e ball taking $(0,1)$ to $0$; use the same symbol
for its continuous ideal extension. Define
\[
 E_n=\Psi(Y_n),\qquad C_n=\CH(E_n),\qquad S_n=\partial C_n,
 \qquad \widehat J_n=\Psi J_n.
\]
Proposition~\ref{prop:reflection-cover} gives, for every fixed $R>0$
and all sufficiently large $n$, an isometry
\begin{equation}\label{eq:lattice-limits-balls}
 \widehat J_n:
 \overline B_{\Sig}(o,R)\longrightarrow
 \overline B_{S_n}(0,R).
\end{equation}
These are the entire balls for the full surface distances. Every path
in the target ball pulls back through this identification. In particular,
the equality rules out shorter comparison paths through portions of the
reflection cover outside the source patch. The source balls are compact by
Lemma~\ref{lem:lattice-limits-source}.

For each fixed $v\in\Lam$, put $z_n(v)=M_n x_n(v)$ for all sufficiently
large $n$. This label is eventually an original interior vertex, so
Proposition~\ref{prop:reflection-cover} also identifies its complete
lifted cusp with peripheral degree one.

\emph{A proper limit of dimension at least two.}
Compactness of the space of nonempty closed subsets of the ideal sphere
gives a further subsequence for which $E_n$ converges in Hausdorff
distance to a compact set $E$. Let
\[
 \Phi(\xi)=\frac{2\xi}{1+|\xi|^2},\qquad
 K_n=\overline{\Phi(C_n)},\qquad K=\operatorname{conv}(E),
 \qquad C=\CH(E),\qquad Y=\Psi^{-1}(E).
\]
The Euclidean convex hull operation does not increase Hausdorff distance:
approximate the points in each convex combination by points of the other
set. Hence $K_n\to K$. The sets $K_n$ have interior, since $X_n$
contains three noncollinear finite vertices and infinity. Since
$0=\Phi\widehat J_n(o)$ lies on $\partial K_n$, a supporting plane
through $0$ contains $K_n$ in a closed half-ball. Passing to a limit
of unit normals puts $K$ in a closed half-ball and keeps $0\in K$.
Thus $E$, and therefore $Y$, is proper.

Choose $r_0>0$ small enough that the intrinsic disk of radius $r_0$
about $o$ lies in the chosen ideal face. Its hyperbolic area is
$2\pi(\cosh r_0-1)>0$. By \eqref{eq:lattice-limits-balls} its image
is an embedded isometric disk on $S_n$, contained in the ambient
hyperbolic ball of radius $r_0$ about $0$. On this fixed compact ball,
the Euclidean Poincar\'e area element is bounded below by a positive
multiple of hyperbolic area. Both singular values of $D\Phi$, displayed
in the proof of Lemma~\ref{lem:lattice-limits-component}, also have
positive minima there. Consequently
\[
 \Area(\partial K_n)\ge c_0>0
\]
for one constant $c_0$ independent of $n$.

If $K$ had dimension at most one, the bodies $K_n$ would lie in
$\epsilon_n$-neighborhoods of a segment of length $\ell_0\le2$, with
$\epsilon_n\to0$. The enclosing capsule has surface area
$2\pi\ell_0\epsilon_n+4\pi\epsilon_n^2$. Surface area is monotone under
inclusion of compact convex bodies: nearest-point projection onto the
inner body is $1$-Lipschitz and maps the outer boundary onto the inner
boundary. The first claim follows by adding the two projection
variational inequalities; the second follows by extending an outward
supporting normal ray from each inner boundary point. The projection
cannot increase two-dimensional Hausdorff measure. The resulting upper
bound on $\Area(\partial K_n)$ tends to zero, a contradiction. Thus
$\dim K\ge2$.

\emph{Intrinsic convergence in dimension three.}
Suppose first that $K$ has interior. We justify the use of
Theorem~\ref{thm:alexandrov} without assuming that $\partial C$ is
connected. For large $L$, form the compact hyperbolic convex bodies
\[
 C_{n,L}=C_n\cap\overline B_{\HH^3}(0,L),\qquad
 C_L=C\cap\overline B_{\HH^3}(0,L).
\]
They converge in Hausdorff distance. Indeed, a point on the cutting
sphere can first be moved slightly toward the common point $0$ of
all the hulls. Approximate that new point from $C_n$ inside the cutting
ball and then let the displacement tend to zero. The reverse inclusion
follows from closed-set convergence. Taking $L$ to contain a fixed
interior ball of $C$ gives an interior ball in all late $C_{n,L}$.
Their boundary spheres therefore converge, and these compact convex
bodies satisfy the interior hypothesis of Theorem~\ref{thm:alexandrov}.

Let $x,x'$ lie in the same component of the true boundary $\partial C$,
and choose a boundary path from $x$ to $x'$ of length $D$. Take
\[
 L>\dist_{\HH^3}(0,x)+D+4
\]
and large enough to contain the common interior ball. The chosen path
lies strictly inside the cutting ball. For true boundary points
$x_n\to x$ and $x'_n\to x'$, any path from $x_n$ to an artificial cap
has length at least
$L-\dist_{\HH^3}(0,x_n)>D+3$ for large $n$.
Distance convergence on the compact boundary spheres gives paths of
length at most $D+1$ between the approximating endpoints. Their
minimizers cannot touch the cap. The same lower bound prevents any
shorter path on the full boundary from leaving the cutting ball.
The limiting minimizer also avoids the cap. Thus
\begin{equation}\label{eq:lattice-limits-distance}
 d_{S_n}(x_n,x'_n)\longrightarrow d_{\partial C}(x,x')
\end{equation}
within each true boundary component.

Conversely, suppose $x_n\in S_n$ converges and
$d_{S_n}(0,x_n)\le R$. Cut at $L>R+3$, also containing the interior
ball. Paths of length at most $R+1$ from $0$ to $x_n$ stay inside
the cut. The compact-sphere distance theorem gives a limiting path of
length at most $R$; it cannot reach the cap. Hence the limiting point
belongs to the true boundary component containing $0$. These arguments
establish both directions of the pointed distance convergence.

\emph{Intrinsic convergence with both planar sheets.}
Suppose instead that $K$ is two-dimensional, in a Klein plane $\Pi$.
Since $0\in K$, this plane passes through $0$. In this paragraph use
Klein coordinates, with their hyperbolic metric, and put
$C^{\flat}=K\cap B^3$. Project $K_n$ orthogonally onto $\Pi$, obtaining
convex disks $Q_n\to K$. The upper and lower boundaries over the
interior of $Q_n$ are graphs of functions $b_n^+$ and $b_n^-$,
respectively concave and convex. Their absolute heights tend uniformly
to zero. On a compact subset of the relative interior of $K$, convexity
bounds each one-sided directional slope by twice the height bound divided
by the distance to the projected boundary. The slopes therefore
tend to zero there. Graph lifts of fixed interior polygonal paths have
hyperbolic lengths tending to the corresponding planar lengths.

Every finite relative boundary point of $K$ lies in the interior of a
straight boundary segment. Every extreme point of
$\operatorname{conv}(E)$ belongs to $E$: express a point as a finite
convex combination, and extremality forces all its terms to be that
point. Thus an extreme point lies on the ideal sphere. A finite
nonextreme boundary point belongs to a nontrivial segment in a supporting
line. Near such a point we also have short connectors to the graph
lifts of nearby interior points. Along a projected straight segment,
a concave or convex function has at most two monotone pieces. The
Euclidean length of its graph is at most the horizontal segment length
plus twice its height oscillation. A vertical boundary fiber has length
at most the total thickness of $K_n$. All these estimates transfer to
hyperbolic length on a fixed compact ambient ball. First let $n$ tend
to infinity for a fixed nearby interior point, and then let that point
approach the boundary. This gives connectors of arbitrarily small length
on either sheet.

The intrinsic metric of the double of $C^{\flat}$ is explicit. On the
same sheet it equals the planar hyperbolic distance $d_\Pi$. Between
opposite sheets it is
\begin{equation}\label{eq:lattice-limits-double}
 \inf_{c\in\partial_{\mathrm{rel}}C^{\flat}}
       \bigl(d_\Pi(x,c)+d_\Pi(c,x')\bigr).
\end{equation}
Every sheet-changing path meets the seam, giving the lower bound, and
two planar geodesic segments through a seam point give the upper bound.
If $C^{\flat}$ is the whole hyperbolic plane, the infimum is infinite
and the double has two separate plane components.

For the approximating bodies, the upper and lower graph pieces meet
through fibers above $\partial Q_n$. A path changing sheets has a
point on such a fiber. If its length and the distances of its endpoints
from the basepoint remain bounded, that crossing lies in a compact
ambient ball. Its projection has a subsequential limit on the finite
seam of $C^{\flat}$. Ambient distance applied to the two subpaths gives
the lower bound in \eqref{eq:lattice-limits-double}. The same-sheet
lower bound is simply ambient distance. Interior graph lifts and the
short connectors just established give the matching upper bounds:
approximate a seam point by a boundary fiber of $Q_n$, move slightly
into the interior on both sides, and lift the two planar segments.
If there is no finite seam, crossing fibers leave every compact ambient
ball and their intrinsic distances tend to infinity. Thus planar
distance convergence retains the two sheet labels; it never identifies
opposite interior points merely because their ambient coordinates agree.

\emph{A marked isometry onto the reached component.}
In either dimension, let $S_0$ denote the component of the limiting
boundary, with its double convention, containing the tracked basepoint.
Choose a countable dense subset of $\Sig$ containing $o$. The images
of each of its points stay in a compact ambient hyperbolic ball, since
ambient distance is at most intrinsic distance. Extract their limits
diagonally, also recording upper or lower sheet labels in the planar
case. The distance convergence established in
\eqref{eq:lattice-limits-distance} and through
\eqref{eq:lattice-limits-double}, together with
\eqref{eq:lattice-limits-balls}, shows that all pairwise distances are
preserved and that all the limits belong to $S_0$.

The component $S_0$ is complete in its path metric. An intrinsic Cauchy
sequence is ambient Cauchy inside hyperbolic space; local convex
boundary charts, or the two planar copies, identify the intrinsic and
surface topologies near its limit. A connected such neighborhood also
keeps the limit in the same component. The limiting correspondence on
the dense set consequently extends to an isometric embedding
\[
 \widehat J:\Sig\longrightarrow S_0.
\]
It is onto. For $x\in S_0$, choose approximating boundary points $x_n$,
with the prescribed sheet if necessary. Their distances from $0$ are
bounded by the proved distance convergence. Equation
\eqref{eq:lattice-limits-balls} gives $x_n=\widehat J_n(s_n)$ with all
$s_n$ in one compact source ball. Pass to a subsequence $s_n\to s$.
For each point $u$ of the chosen dense set, distance convergence gives
\[
 d_{S_0}(x,\widehat J(u))
    =\lim_n d_{S_n}(x_n,\widehat J_n(u))
    =d_{\Sig}(s,u).
\]
Letting $u\to s$ proves $x=\widehat J(s)$.
In dimension three, finite nets in a compact source set and
\eqref{eq:lattice-limits-balls} also give uniform ambient convergence
$\widehat J_n\to\widehat J$ on that set.

\emph{The whole boundary and exclusion of a planar limit.}
Let $q$ be the conformal homeomorphism of
Lemma~\ref{lem:lattice-limits-source}. Orient $S_0$ so that
$\widehat J$ is orientation preserving. The isometry is then conformal,
and $\widehat Jq:\CC\setminus\Lam\to S_0$ satisfies the hypotheses
of Lemma~\ref{lem:lattice-limits-component}. That lemma gives
$S_0=\partial C$ and a homeomorphism
\[
 H:\Chat\longrightarrow\overline S,
 \qquad H|_{\CC\setminus\Lam}=\Phi\widehat Jq,
 \qquad H(\Lam\cup\{\infty\})=E,
\]
where $\overline S$ is the compact Klein boundary sphere, retaining
the double if planar. Put $\eta(v)=H(v)$ and
$y(v)=\Psi^{-1}(\eta(v))$ for $v\in\Lam\cup\{\infty\}$.
This already gives the asserted bijection, except for identifying the
coordinate of the residual end.

If $K$ were planar, the whole-plane case would contradict the
connectedness just proved. In every other planar case sheet interchange
$\mathcal R$ is an orientation-reversing intrinsic isometry of the
double, extending to $\overline S$ and fixing every ideal point.
Consequently
\[
 A=q^{-1}\widehat J^{-1}\mathcal R\widehat Jq
\]
is an antiholomorphic self-homeomorphism of $\CC\setminus\Lam$.
The homeomorphism $H$ extends it to the source sphere and shows that
it fixes every lattice puncture and the source point at infinity.
At each puncture, continuity and a local target coordinate give the
boundedness required by Theorem~\ref{thm:removable}
\cite[Theorem~1.31]{mcmullen2025}, applied after conjugation.
The same argument in reciprocal coordinates applies at infinity.
Thus $A$ is an antiholomorphic sphere automorphism fixing infinity.
Its holomorphic conjugate has a single simple pole there, since a
homeomorphism has local multiplicity one. Its Laurent expansion therefore
gives $A(z)=\alpha\overline z+\beta$ with $\alpha\ne0$.
Fixing $0$ and $1$ gives $\beta=0$ and $\alpha=1$, whereas complex
conjugation does not fix $\om$. This contradiction proves $\dim K=3$.
We have therefore obtained uniform convergence on intrinsic compact
sets to the onto isometry
$J=\Psi^{-1}\widehat J:\Sig\to\partial\CH(Y)$.
Let $\widehat q:\Chat\to\widehat\Sig$ be the extension of $q$ across
the finite cusps and the residual source end. By
Lemma~\ref{lem:lattice-limits-source} and the definition of the filled
source compactification, this is a homeomorphism. For the fixed model
change $\Phi\Psi$ we have $K_Y=K$, and
\[
 \overline J=H\widehat q^{-1}:\widehat\Sig\longrightarrow\partial K_Y
\]
is the required homeomorphic extension of $\Phi\Psi J$.

\emph{Persistence of the individual cusp disks.}
Choose $b$ and $r>0$ such that $\overline B(b,r)$ lies in $K$ and in
all sufficiently late $K_n$. Radial projection
\[
 \rho_b(p)=\frac{p-b}{|p-b|}
\]
identifies $\partial K$ and each $\partial K_n$ with one fixed
direction sphere. Their inverse radial parametrizations converge
uniformly. Indeed, the common inner ball bounds the Lipschitz constants
of their gauges as in Lemma~\ref{lem:lattice-limits-component};
Hausdorff convergence gives convergence of the radial functions,
uniformly by these bounds. The direct projections have a common
continuity bound because all radial distances are at least $r$.
These identifications compare the entire compact boundary spheres.
Moreover, $\rho_b$ restricted to the ideal unit sphere is a fixed
homeomorphism independent of $n$.

Fix $v\in\Lam$ and a simple horocycle $\gamma$ in its standard
complete source cusp. For large $n$ the entire cusp lies in the
lifted source patch. Its image under $\widehat J_n$ is an embedded
peripheral cusp with degree one. On the compact boundary $\partial K_n$,
the curve $\Phi\widehat J_n(\gamma)$ therefore bounds a closed
Jordan disk whose only ideal point is $\Psi(z_n(v))$.
This is the disk on the cusp side: the finite canonical marked isometry
extends across the puncture, and its lift of degree one gives a disk
with that one puncture. The reflection cover includes all target
ideal points, so the exclusion applies to the entire set $E_n$.
Fix two distinct source points outside the cusp; their images select
the other side of the peripheral curve.

In the common direction sphere the parametrized Jordan curves
$\rho_b\Phi\widehat J_n(\gamma)$ converge uniformly to
$\rho_b\Phi\widehat J(\gamma)$. Every compact subset of either
complementary disk of the limiting curve lies on the corresponding
side of the approximating curves for all large $n$. For a direct
justification, choose a sphere chart whose omitted point is away from
the limiting curve. On a compact set away from that curve, a sufficiently
small uniform perturbation of its parametrization gives a homotopy
avoiding every point of the set. Winding numbers, and hence sides,
are unchanged. The fixed exterior reference points determine the
cusp side throughout.

Take successively deeper source horocycles at $v$. Under the
homeomorphic compactification $H$, their closed cusp disks shrink to
$\eta(v)$. Given a neighborhood of $\rho_b(\eta(v))$ in the direction
sphere, choose one of these disks and a slightly larger topological
disk whose closure is contained in the neighborhood. The exterior of
the larger disk is connected and disjoint from the approximating
curves for large $n$, and lies on their exterior side. The corresponding cusp disks are consequently
contained in the chosen neighborhood.
Their unique ideal points satisfy
\[
 \Psi(z_n(v))\longrightarrow\eta(v),\qquad
 z_n(v)\longrightarrow y(v).
\]

Now fix just one limiting cusp disk. Choose a spherical neighborhood
$U_v$ of $y(v)$ so small that the closure of
$\rho_b\Psi(U_v)$ lies strictly inside this disk in the direction
sphere. Side stability places this closure inside the corresponding
approximating cusp disks for all large $n$. Since each such disk
contains exactly one ideal point and that point converges to $y(v)$,
\begin{equation}\label{eq:lattice-limits-singleton}
 Y_n\cap U_v=\{z_n(v)\}\qquad\text{for all sufficiently large }n.
\end{equation}
The neighborhood is fixed, and the equality concerns every ideal
point of the reflected target.

\emph{The residual end is infinity.}
Every $Y_n$ contains infinity, so $\infty\in Y$. For finite $v$, the
point $z_n(v)$ is finite. If $y(v)=\infty$, the fixed neighborhood
in \eqref{eq:lattice-limits-singleton} would contain both infinity
and the distinct point $z_n(v)$ for all large $n$, a contradiction.
The bijection already obtained therefore forces $y(\infty)=\infty$.
The compactification makes this the only possible accumulation point
of $Y$; it is an accumulation point because $\Lam$ is infinite.

For completeness, the singleton assertion also excludes untracked
vertices in any fixed bounded target disk. Otherwise choose ideal
points in that disk avoiding the approximating images of an increasing
exhaustion of $\Lam$. This includes points outside the chosen source
half, to which no label in the fixed lattice patch has been assigned.
A subsequence has a finite ideal limit in $Y$, necessarily
$y(v)$ for some fixed $v$. Equation
\eqref{eq:lattice-limits-singleton} forces those vertices eventually
to equal $z_n(v)$, a contradiction. Thus the whole-ball identification,
the onto isometry, and the fixed cusp neighborhoods retain all the
marking information asserted in the theorem.
\end{proof}

\section{Packing and logarithmic lattice rigidity}\label{sec:06-logarithmic-rigidity}

We keep the lattice $\Lam$, its coordinate distance $N$, and the
standard decoration of $\Sig$. We first obtain logarithmic control of
vertex scales from a planar packing. We then express decorated cusp
angles by Voronoi envelopes and prove rigidity under this control.

\begin{theorem}\label{thm:packing-growth}
Let $y:\Lam\to\CC$ be injective and proper, meaning that every bounded
planar set has finite inverse image, and let $r_v>0$ satisfy
$r_v\le3|y(v)-y(w)|$ for every $v\ne w$ and
$|y(v)-y(w)|\le3\min(r_v,r_w)$ on every original lattice edge.
Then, with the universal constant $C=10^9$, for every distinct pair,
\[
 \frac{|y(v)-y(w)|}{r_w}\le(N(v,w)+2)^C,
 \qquad |\log r_v-\log r_w|\le C\log(N(v,w)+2).
\]
\end{theorem}

The all-pairs separation hypothesis supplies disjoint disks. The edge
hypothesis lets us estimate the energy of a radial cutoff on the original
lattice.

\begin{proof}
We construct a cutoff with small energy, prove the needed resistance
estimate by Fourier inversion, and reverse the choice of root to control
both ratios of scales.

\emph{A packing controls the cutoff energy.}
Fix a root $w$. Translating the lattice and replacing all target
coordinates and scales by
\[
 \frac{y(v+w)-y(w)}{r_w},\qquad \frac{r_{v+w}}{r_w}
\]
preserve the hypotheses. We may therefore take $w=0$, $y(0)=0$, and
$r_0=1$. Put $s_p=|y(p)|$ for $p\in\Lam$. For $p\ne0$ the separation
hypothesis gives $s_p\ge1/3$ and $r_p\le3s_p$. The open disks
\[
 \mathcal B_p=B(y(p),r_p/6)
\]
are pairwise disjoint, since
$(r_p+r_q)/6\le|y(p)-y(q)|$ for distinct $p,q$.
On an edge $pq$ with neither endpoint equal to $0$,
\[
 s_p\le s_q+3r_q\le10s_q,
 \qquad s_q\le10s_p.
\]
On an edge $0p$ we instead have $s_p\le3$ and $s_p\le3r_p$.

For $R\ge2$ define
\[
 \chi_R(p)=\max\left(0,\min\left(1,
       1-\frac{\log\max(1,s_p)}{\log R}\right)\right).
\]
Properness makes $\chi_R$ finitely supported, and $\chi_R(0)=1$.
Truncating the logarithm cannot increase differences. Integrating its
derivative between $s_p$ and $s_q$ gives
\[
 |\chi_R(p)-\chi_R(q)|
 \le\frac{|s_p-s_q|}
 {\max(1,\min(s_p,s_q))\log R}.
\]
For a nonroot edge, the radial ratio bound and
$|s_p-s_q|\le3r_p$ imply
\begin{equation}\label{eq:logarithmic-rigidity-cutoff-edge}
 |\chi_R(p)-\chi_R(q)|
 \le\frac{30r_p}{\max(1,s_p)\log R}.
\end{equation}
This holds with either endpoint in place of $p$. On a root edge $0p$,
the difference is at most $\log3/\log R$, which proves
\eqref{eq:logarithmic-rigidity-cutoff-edge} at the root. At the other
endpoint the preceding radial estimate bounds the difference by
$3r_p/\log R$. Since $\max(1,s_p)\le3$, this is at most
$9r_p/(\max(1,s_p)\log R)$, so the same bound holds there as well.

For a finitely supported real function $\varphi$ on $\Lam$, let
\[
 \mathcal E(\varphi)
 =\sum_{\{p,q\}\text{ an original edge}}
       (\varphi(p)-\varphi(q))^2,
\]
where each unoriented edge is counted once. A nonzero cutoff difference
has an endpoint with $s_p<R$. Both endpoints then have radial coordinate
at most $10R$, by the preceding edge estimates, including root edges.
Assign each contributing edge to either endpoint. Each vertex receives
at most six edges, and hence
\[
 \mathcal E(\chi_R)
 \le\frac{5400}{(\log R)^2}
     \sum_{s_p\le10R}\frac{r_p^2}{\max(1,s_p)^2}.
\]

If $p\ne0$ and $z\in\mathcal B_p$, then $|z|\le3s_p/2$.
The root disk has radius $1/6$, so in every case
\[
 \max(1,|z|)\le\frac32\max(1,s_p),\qquad
 \int_{\mathcal B_p}\frac{\dd\Area(z)}{\max(1,|z|)^2}
 \ge\frac{\pi r_p^2}{81\max(1,s_p)^2}.
\]
All disks with $s_p\le10R$ lie in $B(0,15R)$. Disjointness and polar
integration give
\begin{align*}
 \sum_{s_p\le10R}\frac{r_p^2}{\max(1,s_p)^2}
 &\le\frac{81}{\pi}\int_{B(0,15R)}
                       \frac{\dd\Area(z)}{\max(1,|z|)^2}\\
 &=81\bigl(1+2\log(15R)\bigr)
 \le1000\log R.
\end{align*}
For the last inequality, $\log2>1/2$ and $\log15<4\log2$ yield
\[
 81\left(2+\frac{1+2\log15}{\log R}\right)
 <81(2+2+8)=972<1000.
\]
We have proved
\begin{equation}\label{eq:logarithmic-rigidity-cutoff-energy}
 \mathcal E(\chi_R)\le\frac{5{,}400{,}000}{\log R}
                     \le\frac{10^7}{\log R}.
\end{equation}

\emph{Fourier inversion bounds resistance on the original lattice.}
Identify $m+n\om$ with $(m,n)\in\ZZ^2$. The horizontal and vertical
edges give the energy
\[
 \mathcal E_{\square}(\varphi)
 =\sum_{(m,n)\in\ZZ^2}\left(
   (\varphi(m+1,n)-\varphi(m,n))^2
  +(\varphi(m,n+1)-\varphi(m,n))^2\right)
 \le\mathcal E(\varphi).
\]
For finitely supported $\varphi$, its Fourier transform is the finite sum
\[
 \widehat\varphi(\xi)
 =\sum_{\mathbf p\in\ZZ^2}\varphi(\mathbf p)
                    e^{-\mathrm i\xi\cdot\mathbf p},
 \qquad \xi\in[-\pi,\pi]^2.
\]
Orthogonality of integer exponentials gives inversion and
\[
 \mathcal E_{\square}(\varphi)
 =\frac1{(2\pi)^2}\int_{[-\pi,\pi]^2}
      \mathcal D(\xi)|\widehat\varphi(\xi)|^2\dd\xi,
 \qquad
 \mathcal D(\xi)=4\sin^2(\xi_1/2)+4\sin^2(\xi_2/2).
\]
Fix $v=m+n\om\ne0$, write $\mathbf v=(m,n)$, and put
$\nu=N(v,0)$. Weighted Cauchy--Schwarz applied to Fourier inversion gives
\[
 |\varphi(v)-\varphi(0)|^2
 \le\mathcal E_{\square}(\varphi)\,
 \frac1{(2\pi)^2}\int_{[-\pi,\pi]^2}
     \frac{|e^{\mathrm i\xi\cdot\mathbf v}-1|^2}
          {\mathcal D(\xi)}\dd\xi.
\]
The quotient is bounded near $\xi=0$. On the square,
\[
 \mathcal D(\xi)\ge\frac4{\pi^2}|\xi|^2,
 \qquad
 |e^{\mathrm i\xi\cdot\mathbf v}-1|
 \le\min(2,\nu|\xi|).
\]
Replace the integrand by this radial upper bound, enlarge the integration
region to the disk of radius $\sqrt2\pi$, and split at $1/\nu$. The
integral with its prefactor is at most
\[
 \frac\pi8\left(
    \int_0^{1/\nu}\nu^2\rho\dd\rho
   +\int_{1/\nu}^{\sqrt2\pi}\frac4\rho\dd\rho\right)
 =\frac\pi8\left(\frac12+4\log(\sqrt2\pi\nu)\right)
 \le10\log(\nu+2).
\]
Indeed, $\pi<4$, $\sqrt2\pi<8$, and $\log2<1$ bound the middle
expression by $25/4+2\log\nu$, which is less than
$10\log(\nu+2)$ for $\nu\ge1$. Thus
\begin{equation}\label{eq:logarithmic-rigidity-resistance}
 |\varphi(v)-\varphi(0)|^2
 \le10\log(N(v,0)+2)\mathcal E(\varphi).
\end{equation}

\emph{Displacement and both scale ratios follow.}
If $s_v\ge2$, take $R=s_v$. The cutoff is $1$ at $0$ and $0$ at $v$,
so \eqref{eq:logarithmic-rigidity-cutoff-energy} and
\eqref{eq:logarithmic-rigidity-resistance} imply
\[
 \log s_v\le10^8\log(N(v,0)+2).
\]
The same bound holds when $s_v<2$. It gives the claimed displacement
estimate, with exponent $10^8$. Since $r_v\le3s_v$,
\[
 \log\frac{r_v}{r_0}
 \le\log3+10^8\log(N(v,0)+2).
\]
Repeat the argument with root $v$, translating the target by $-y(v)$
and dividing all target distances and all scales by $r_v$. Properness,
the two inequalities, and the original lattice are preserved, with the
same constants. We obtain the same upper bound for
$\log(r_0/r_v)$. These two bounds control the absolute logarithmic
difference. Since $\log3\le\log(N(v,0)+2)$ and the initial root was
arbitrary, both conclusions hold with $C=10^9$.
\end{proof}

We next describe the angle comparison for decorations. Its strictness
depends on arcs of positive length in the Voronoi boundary.

\begin{definition}\label{def:logarithmic-rigidity-voronoi}
Let $S$ be a complete marked hyperbolic surface with cusp set $V$ and a
decoration of circumferences $w_i>0$, $i\in V$. In the universal cover,
write $B_s$ for the signed Busemann function of a lifted horocycle. Its
\emph{decorated Voronoi cell} is
\[
 \mathcal V_s=\{z:B_s(z)\le B_t(z)\text{ for every lifted horocycle }t\}.
\]
A \emph{root chart} sends the chosen cusp lift to infinity and its
horocycle to height $1$. Its primitive stabilizer is $x\mapsto x+L$;
we write $(a_s,d_s)$ for the centers and diameters of the other lifted
horocycles in this chart.
The \emph{Voronoi dual}, when it is an ideal cell decomposition, has an
ideal vertex at each cusp, an edge dual to each common Voronoi arc of
positive length, and a polygon dual to each Voronoi vertex. The
\emph{cusp star} is the collection of dual cells incident to that cusp.
When the lambda lengths on a triangulation of these cells form Euclidean
triangles, the sum of their angles at a cusp is its \emph{Euclidean root
angle}. For two decorations indexed by $k=0,1$, we denote their
circumferences and root angles by $w_i^k$ and $\theta_i^k$.
\end{definition}

\begin{proposition}\label{prop:envelope-comparison}
Let $S$ be one complete connected marked hyperbolic surface with a finite
or countable cusp set $V$. Let $D_0,D_1$ be two decorations with cusp
circumferences $w_i^0,w_i^1$ and corresponding Euclidean angle sums
$\theta_i^0,\theta_i^1$. Assume that each decorated Voronoi dual is a cell decomposition
with finite cusp stars and connected genuine adjacency. In each root
chart normalized to horocycle $y=1$ with period $L$, assume its envelope
$F(x)=\sup_s(d_s-(x-a_s)^2)$ is positive and is the maximum of finitely
many of these quadratic functions on a period. The Euclidean root angle associated to the
lambda lengths, independent of polygon diagonals, is
\[
 \theta=\int_0^L F(x)^{-1/2}\dd x.
\]
If $v_i=\log(w_i^0/w_i^1)$ has a global maximum at $i$, then
$\theta_i^1\le\theta_i^0$; equality forces $v_j=v_i$ at each genuine
neighbor for the first decoration. If all corresponding angles agree
and $v$ attains a global maximum, $v$ is constant.
\end{proposition}

The integral expresses the angle using the boundary arcs themselves.
It also identifies exactly which neighbors must share an attained
maximum.

\begin{proof}
We compute the angle at each switch of the envelope, sum these angles
over one period, and use monotonicity of the integral to propagate a
maximum.

\emph{An envelope switch gives a Euclidean angle.}
First fix one decoration and one root chart. Write $(x,t)$ for the
upper-half-plane coordinates. The root cell is determined by
\[
 -\log t\le
 \log\frac{(x-a_s)^2+t^2}{d_st}
 \quad\text{for all }s,
\]
or equivalently $t^2\ge F(x)$. Its boundary is therefore the graph
$t=\sqrt{F(x)}$. Positivity and the finite-family hypothesis make it a
finite union of upper circular arcs on a period, at heights bounded
above and away from zero.

Suppose consecutive active arcs have centers $a<b$ and parameters
$d_a,d_b$. The order follows from
\[
 \frac{\mathrm d}{\mathrm dx}
 \bigl(d_b-(x-b)^2-d_a+(x-a)^2\bigr)=2(b-a):
\]
an arc that becomes dominant to the right has the larger center. At
their switch $(x_*,t_*)$,
\[
 d_a=(x_*-a)^2+t_*^2,\qquad
 d_b=(x_*-b)^2+t_*^2,
 \qquad t_*>0.
\]
The signed horocycle distances on the ideal edges give the lambda lengths
\begin{equation}\label{eq:logarithmic-rigidity-lambda-lengths}
 \lambda_{\infty a}=d_a^{-1/2},\qquad
 \lambda_{\infty b}=d_b^{-1/2},\qquad
 \lambda_{ab}=\frac{b-a}{\sqrt{d_ad_b}}.
\end{equation}
For the first two formulas, the signed vertical gaps are
$\log(1/d_a)$ and $\log(1/d_b)$. For the last, the map
$z\mapsto(z-a)/(b-z)$ takes the two horocycles to a circle of diameter
$d_a/(b-a)$ at $0$ and a horizontal horocycle of height $(b-a)/d_b$.
Their signed gap is $\log((b-a)^2/(d_ad_b))$.

The circles of radii $\sqrt{d_a}$ and $\sqrt{d_b}$ about $a,b$ intersect
at positive height. Hence
\[
 |\sqrt{d_a}-\sqrt{d_b}|<b-a<\sqrt{d_a}+\sqrt{d_b},
\]
which gives the strict triangle inequalities for
\eqref{eq:logarithmic-rigidity-lambda-lengths}. Set
\[
 \alpha_a=\arg(x_*-a+\mathrm i t_*),\qquad
 \alpha_b=\arg(x_*-b+\mathrm i t_*),
\]
with arguments in $(0,\pi)$. Then $\alpha_b>\alpha_a$, and
\[
 \cos(\alpha_b-\alpha_a)
 =\frac{d_a+d_b-(b-a)^2}{2\sqrt{d_ad_b}}.
\]
The cosine law applied to the three lambda lengths gives this same
expression for the root angle. Thus that angle is the upward jump
$\alpha_b-\alpha_a$.

Multiple ties give the same conclusion for any polygon diagonals.
For every center $c$ tied at $(x_*,t_*)$, define
\[
 \zeta_\infty=0,\qquad
 \zeta_c=\frac1{x_*-c-\mathrm i t_*}.
\]
Here $d_c=(x_*-c)^2+t_*^2$. All tied centers lie between the one-sided
active centers, because differences of the envelope quadratics are
affine functions of $x$.
These points satisfy
$\operatorname{Im}\zeta=t_*|\zeta|^2$, so they lie on one circle.
Moreover,
\[
 |\zeta_c|=d_c^{-1/2},\qquad
 |\zeta_c-\zeta_{c'}|
   =\frac{|c-c'|}{\sqrt{d_cd_{c'}}}.
\]
They realize the dual polygon as a convex Euclidean cyclic polygon with
the required lambda lengths, including its diagonals. A fan from the
root has successive angle differences that telescope from the first
active center to the last. Every triangulation has the same polygon
angle at the root. A center tied only at this switch is included in this
polygon, even though it has no boundary arc of positive length.

\emph{The jumps sum to the angle integral.}
On an interval where $a$ is active, put
\[
 \alpha(x)=\arg\bigl(x-a+\mathrm i\sqrt{d_a-(x-a)^2}\bigr).
\]
The identities $x-a=\sqrt{d_a}\cos\alpha$ and
$\sqrt{F(x)}=\sqrt{d_a}\sin\alpha$ imply
$\alpha'(x)=-F(x)^{-1/2}$. At a switch its upward jump is the root angle
just computed. Choose the endpoints of a period away from switches.
The active center at the right endpoint is the translate by $L$ of the
one at the left, so the endpoint arguments agree. The net change is
zero, giving
\[
 0=-\int_0^L F(x)^{-1/2}\dd x
       +\sum_{\text{switches}}\text{root angle}.
\]
The sum is the entire cusp angle, since the finite cusp star is dual to
these switches. This proves the formula and independence of diagonals.

\emph{Increasing a genuine neighbor's diameter makes the inequality
strict.}
Consider two finite positive envelopes with the same centers and period
and with $d_s^1\ge d_s^0$ for every lift. Then $F_1\ge F_0$, so their
angle integrals satisfy $\theta^1\le\theta^0$. If $s$ is active on an
open interval for $F_0$ and $d_s^1>d_s^0$, then throughout that interval
\[
 F_1(x)\ge d_s^1-(x-a_s)^2
          >d_s^0-(x-a_s)^2=F_0(x).
\]
The integral inequality is then strict. Equality of the angles therefore
forces equality of every diameter active on an open interval for the
first envelope.

\emph{A maximum propagates along genuine adjacency.}
Suppose $v_i=m=\max_jv_j$. Multiply all circumferences of the second
decoration by $e^m$. Every signed Busemann function decreases by $m$;
the cells stay fixed and every lambda length is multiplied by $e^{-m}$.
Thus the angles stay fixed as well. The two root horocycles at $i$ now
agree, so we may use one chart with root height $1$ and period $L$.
If lift $s$ has label $j(s)$, its diameter in the rescaled second
decoration is
\[
 d_s^{1'}=e^{m-v_{j(s)}}d_s^0\ge d_s^0.
\]
Here the diameter ratio equals the circumference ratio: inversion at
the finite center turns diameter $d$ into horizontal height $1/d$,
and circumference is inversely proportional to that height. The marked
surface fixes the centers in this common chart.

The envelope comparison gives $\theta_i^1\le\theta_i^0$. Equality
forces $v_{j(s)}=m$ for every lift active on an open boundary arc, hence
for every genuine neighbor of $i$ in the first decoration. If all angles
agree, each such neighbor is another point where $v$ attains its global
maximum. Repeating the argument along finite paths in the connected
genuine adjacency graph proves $v\equiv m$.
\end{proof}

We apply this comparison on the whole lattice surface. The logarithmic
oscillation bound gives a finite set of possible competitors in each
root chart, uniformly under translations.

\begin{theorem}\label{thm:log-liouville}
Let $u:\Lam\to\RR$ satisfy
$|u_v-u_w|\le A\log(N(v,w)+2)$ for every pair and one finite $A\ge0$.
Change the standard lattice cusp circumferences to $6e^{-u_v}$.
In the standard root chart, if $(a_s,d_s)$ are the other standard lifted
horocycles with labels $j(s)$, set
\[
 F_i(u;x)=\sup_s\bigl(e^{u_i-u_{j(s)}}d_s-(x-a_s)^2\bigr).
\]
These envelopes have uniform positive lower and finite upper bounds,
and on each period are maxima of a fixed finite family depending only on
$A$ and the translated standard chart. Their angles
$\theta_i(u)=\int_0^6F_i(u;x)^{-1/2}\dd x$
are continuous under pointwise convergence with the same $A$.
The decorated diagram is locally finite with connected genuine adjacency.
If $\theta_i(u)=2\pi$ for every $i$, then $u$ is constant.
\end{theorem}

The pairwise form of the hypothesis is preserved when we translate the
lattice and subtract the value at the new root. This permits an attained
maximum argument for a translated difference of $u$.

\begin{proof}
We first relate the standard horoball diameters to their lattice labels.
This gives uniform envelope bounds, a finite set of competitors, and
connected adjacency. We then take a translated limit in which an edge
increment attains its supremum and apply
Proposition~\ref{prop:envelope-comparison}.

\emph{Standard horoball diameters control label distance.}
The three standard horocycle sectors in an original ideal triangle have
corner arc length $1$ and are pairwise tangent along its sides. They
glue to disjoint embedded cusp neighborhoods on $\Sig$. Their lifts are
therefore disjoint standard open horoballs. The complement of the three
sectors in one face is a compact curvilinear triangle. On this compact
set, the minimum of the three signed Busemann functions has a finite
maximum. Enlarging it to a nonnegative constant $C_0$ gives the following
uniform covering property: at every point of the universal cover, some
standard lifted horocycle has signed distance at most $C_0$. The same
constant works in every face, and within a sector its own signed
distance is nonpositive.

Fix root $i$ and normalize its standard horocycle to height $1$. The
primitive period is $6$. Every other standard horoball is a disk tangent
to the real axis with diameter $0<d_s\le1$, because it is disjoint from
the root horoball $t>1$.

Each original ideal face of $\Sig$ lies on a hemisphere of radius
$1/\sqrt3$. The upper-half-space metric divides Euclidean lengths by
height, so vertical projection to the lattice plane satisfies
\[
 |\dd z|\le\frac1{\sqrt3}\dd s_{\HH}.
\]
The resulting length bound holds across face edges as well. A standard
cusp section projects into the six closed unit triangles incident to its
label, hence within Euclidean distance $1$ of that label. The vertical
segment from $(a_s,1)$ to $(a_s,d_s)$ joins the two standard lifted
horocycles and has hyperbolic length $\log(1/d_s)$. Projecting the segment
to $\Sig$ and then to the lattice plane gives
\[
 |j(s)-i|\le2+\frac1{\sqrt3}\log(1/d_s).
\]
For lattice coordinates, the inequality
$(|m|+|n|)^2\le4(m^2+mn+n^2)=4|m+n\om|^2$ gives
$N(v,w)\le2|v-w|$. Consequently
\begin{equation}\label{eq:logarithmic-rigidity-label-diameter}
 N(j(s),i)+2\le6\bigl(1+\log(1/d_s)\bigr).
\end{equation}
This estimate also applies when distinct lifts have the same cusp label.

\emph{The changed envelopes have uniform upper and lower bounds.}
Multiply every changed circumference by $e^{u_i}$. This preserves cells
and angles and restores the root horocycle to height $1$ in its standard
chart. The other diameters become
\[
 d_s(u)=e^{u_i-u_{j(s)}}d_s.
\]
The logarithmic hypothesis and
\eqref{eq:logarithmic-rigidity-label-diameter} imply
\begin{equation}\label{eq:logarithmic-rigidity-diameter-bound}
 6^{-A}d_s\bigl(1+\log(1/d_s)\bigr)^{-A}
 \le d_s(u)
 \le 6^A d_s\bigl(1+\log(1/d_s)\bigr)^A.
\end{equation}
For $0<d\le1$ put
\[
 H_A(d)=6^A d\bigl(1+\log(1/d)\bigr)^A,
 \qquad D_A=\sup_{0<d\le1}H_A(d).
\]
With $\tau=\log(1/d)$ we have
$H_A(d)=6^Ae^{-\tau}(1+\tau)^A\to0$ as $\tau\to\infty$.
Thus $D_A$ is finite and positive, and $F_i(u;x)\le D_A$.

Choose a finite $T$ satisfying
\[
 T>C_0+1,\qquad
 T>C_0+A\log\bigl(6(1+C_0+T)\bigr)+1,
\]
and set $\varepsilon=e^{-T}$. Such a $T$ exists since the logarithm is
sublinear in $T$. At $(x,\varepsilon)$ the standard root has signed
distance $T>C_0$. Hence the covering property supplies another lift with
\[
 B_{a_s,d_s}(x,\varepsilon)
 =\log\frac{(x-a_s)^2+\varepsilon^2}{d_s\varepsilon}
 \le C_0.
\]
It follows that $d_s\ge e^{-C_0}\varepsilon$ and
$\log(1/d_s)\le C_0+T$. For this lift the changed Busemann value is
\begin{align*}
 B_{a_s,d_s(u)}(x,\varepsilon)
 &=B_{a_s,d_s}(x,\varepsilon)+u_{j(s)}-u_i\\
 &\le C_0+A\log\bigl(6(1+C_0+T)\bigr)<T-1.
\end{align*}
Substituting $\varepsilon=e^{-T}$ in this inequality yields
\[
 d_s(u)>e\bigl((x-a_s)^2+\varepsilon^2\bigr),
 \qquad
 d_s(u)-(x-a_s)^2>e\varepsilon^2.
\]
We have proved the uniform bounds
\begin{equation}\label{eq:logarithmic-rigidity-envelope-bounds}
 0<c_A:=e^{1-2T}\le F_i(u;x)\le D_A<\infty.
\end{equation}
The constants depend only on $A$ and the standard geometry. The
disjointness used above concerns the standard horoballs; the changed
horocycles may overlap.

\emph{Only a fixed finite family can contribute.}
Choose $\eta\in(0,1)$ such that $H_A(d)<c_A/2$ for $0<d<\eta$.
A lift with $d_s<\eta$ has envelope value below $c_A/2$ everywhere.
For $x\in[0,6]$, a center outside
\[
 [-\sqrt{D_A},\,6+\sqrt{D_A}]
\]
has negative envelope value. All remaining standard horoballs have
diameter at least $\eta$ and lie in the rectangle
\[
 [-\sqrt{D_A}-1/2,\,6+\sqrt{D_A}+1/2]\times[0,1].
\]
Their disjoint interiors have Euclidean area at least $\pi\eta^2/4$
each. There are therefore at most
\[
 \frac{4(7+2\sqrt{D_A})}{\pi\eta^2}
\]
such disks. Since the omitted terms are uniformly below $c_A$, the full
supremum is the maximum over this finite family. The family is fixed by
$A$ and the standard chart, independently of $u$.

In particular $F_i(u;\cdot)$ is continuous and $6$-periodic. If
$\widetilde u^{(n)}\to\widetilde u$ pointwise with the same constant $A$, then on
the fixed family
\[
 \|F_i(\widetilde u^{(n)};\cdot)-F_i(\widetilde u;\cdot)\|_\infty
 \le\max_s d_s\left|
 e^{\widetilde u_i^{(n)}-\widetilde u_{j(s)}^{(n)}}
       -e^{\widetilde u_i-\widetilde u_{j(s)}}\right|\longrightarrow0.
\]
The lower bound in \eqref{eq:logarithmic-rigidity-envelope-bounds} gives
\[
 |\theta_i(\widetilde u^{(n)})-\theta_i(\widetilde u)|
 \le3c_A^{-3/2}
    \|F_i(\widetilde u^{(n)};\cdot)-F_i(\widetilde u;\cdot)\|_\infty
 \longrightarrow0.
\]
Lattice translations preserve $\Sig$ and its standard decoration. Thus
translated root charts have the same finite pattern of label offsets,
and all these bounds hold uniformly under translation.

\emph{The diagram is locally finite and its genuine adjacency is
connected.}
Use a fixed standard root chart with the changed root restored to height
$1$. Let $K$ be a compact set in the universal cover, contained in
\[
 \{(x,t):|x|\le R_0,\quad r_0\le t\le Y_0\},
 \qquad 0<r_0\le Y_0<\infty.
\]
Choose $M$ bounding the root Busemann value $-\log t$ above on $K$.
A finite-center lift whose changed Busemann value is at most $M+1$
somewhere on $K$ satisfies
\[
 d_s(u)\ge r_0e^{-(M+1)},\qquad
 |a_s|\le R_0+\sqrt{e^{M+1}D_AY_0}.
\]
Both inequalities follow directly from
$((x-a_s)^2+t^2)/(d_s(u)t)\le e^{M+1}$.
Since $H_A(d)\to0$ as $d\downarrow0$, the first inequality and
\eqref{eq:logarithmic-rigidity-diameter-bound} bound $d_s$ away from
zero. The second bounds its center. Disjoint standard disk area again
leaves only finitely many competing lifts. The root is always a
competitor with value at most $M$, so the minimum Busemann value is
attained on $K$, and the diagram is locally finite.

Each cell is convex: a comparison between two Busemann functions is a
hyperbolic half-plane inequality, as is seen by moving one center to
infinity. In its own normalized root chart every cell contains
$t>\sqrt{D_A}$, so it has nonempty interior. The cells cover the
universal cover. Join interior points of any two cells by a compact
path. In a compact neighborhood of this path only finitely many cells
compete. Perturb the path to avoid their vertices and to cross their
sides transversely. It then crosses finitely many arcs of positive
length and gives a finite chain of genuine neighbors. Projection to
cusp labels proves connectedness.

For each root, its boundary modulo the parabolic period is the graph
of $\sqrt{F_i}$ in the compact strip
\[
 0\le x\le6,\qquad \sqrt{c_A}\le t\le\sqrt{D_A}.
\]
It has finitely many arcs and switches. Thus there are finitely many
incident dual cells and no ideal endpoint on this boundary. The locally
finite convex Voronoi cells give their dual ideal cell decomposition:
an open side is dual to the edge joining its two incident ideal centers,
and a vertex gives
the polygon of its finitely many incident centers in cyclic order.
To construct the dual, draw arcs through the cell interiors and
straighten them. Their lifts have noninterlacing endpoints on the ideal
circle, so the geodesic representatives do not cross. The finite
boundary description gives finite cusp stars. The hypotheses of
Proposition~\ref{prop:envelope-comparison} now hold. In particular, its
angle formula identifies the integrals above with the Euclidean cusp
angles, independently of polygon diagonals.

\emph{A translated difference attains its maximum.}
Assume all angles equal $2\pi$ and $u$ is nonconstant. The original
lattice graph is connected, so some direction
\[
 e\in\{1,-1,\om,-\om,\om-1,1-\om\}
\]
has a positive increment. Set
\[
 \mu=\sup_{z\in\Lam}(u_{z+e}-u_z)>0.
\]
The pairwise bound gives $\mu\le A\log4<\infty$. Choose $z_n\in\Lam$
so that these increments tend to $\mu$, and define
\[
 u^{(n)}_z=u_{z+z_n}-u_{z_n}.
\]
Each $u^{(n)}$ satisfies the same pairwise bound, has value $0$ at $0$,
and obeys $|u^{(n)}_z|\le A\log(N(z,0)+2)$. A diagonal subsequence on
the countable lattice converges pointwise to a real-valued function
$p$. The pairwise bound passes to $p$.

Translation of the lattice and addition of a constant to the factors
preserve the angles, so every $u^{(n)}$ has all angles equal to $2\pi$.
The proved continuity gives $\theta_i(p)=2\pi$ for every $i$. Put
$q_z=p_{z+e}$. Translation gives the same conclusions for $q$. The
relative logarithmic circumference ratio of these two decorations is
\[
 \delta_z=\log\frac{6e^{-p_z}}{6e^{-q_z}}=q_z-p_z.
\]
Every approximating increment is at most $\mu$, while the increment at
$z_n$ tends to $\mu$. Therefore $\delta_z\le\mu$ for all $z$ and
$\delta_0=\mu$. Proposition~\ref{prop:envelope-comparison}, with the
finite stars and connected adjacency just proved, gives
$\delta_z=\mu$ at every lattice label.

\emph{Linear growth contradicts the logarithmic bound.}
Iterating $p_{z+e}-p_z=\mu$ yields $p_{ke}-p_0=k\mu$ for each positive
integer $k$. The inherited pairwise bound instead gives
\[
 k\mu=|p_{ke}-p_0|
 \le A\log(N(ke,0)+2)
 \le A\log(2k+2).
\]
After division by $k$, the right-hand side tends to zero, contradicting
$\mu>0$. Thus $u$ is constant.
\end{proof}

\section{Microscopic marked similarity}\label{sec:07-microscopic-similarity}

We continue with the standard decoration of $\Sig$ and the canonical maps of
Section~\ref{sec:02-preliminaries}. We show that a source patch whose
boundary recedes to infinity has a similarity as its entire marked limit.
The essential passage is from finite planar angle sums to the angles of
the limiting decoration. We control both the endpoint labels and the
homotopy classes of the horocycles that contribute to those angles.

\begin{theorem}\label{thm:microscopic}
Let $P_n$ be arbitrary marked equilateral disks of mesh $h_n$, with inverse
canonical maps $g_n$. Choose interior vertices $v_n$ and unit complex
numbers $a_n$ so that $Q_n=(P_n-v_n)/(h_n a_n)$ is a lattice subcomplex
containing $\overline B(0,L_n)$, where $L_n\to\infty$.
Let $X_n$ be the reflected target vertex set, $w_n=g_n(v_n)$,
$\delta_n=\dist(w_n,X_n\setminus\{w_n\})$, and
\[
 Z_n=(X_n-w_n)/\delta_n,
 \qquad A_n(z)=\frac{g_n(v_n+h_n a_nz)-w_n}{\delta_n}.
\]
Every subsequence has a further subsequence and $\theta\in\RR$ such that
$Z_n\cup\{\infty\}$ converges in spherical Hausdorff distance to
$e^{i\theta}\Lam\cup\{\infty\}$ and $A_n\to e^{i\theta}\id$ uniformly
on every compact subset of $\CC$, including closed source faces.
For every compact target set $K$, a fixed finite set of source labels
contains all points of $Z_n\cap K$ eventually. Every fixed original face
patch eventually persists as a marked target Delaunay patch. The whole
intrinsic limit is the hyperbolic extension of the same marked similarity.
\end{theorem}

The convergence includes the canonical maps on the edges and interiors
of the original faces. The persistence assertion also specifies the marked
ideal arcs, including their homotopy classes.

\begin{proof}
We first pass to a whole marked limit and record its positive cusp scales.
We then identify one finite family of lifts that computes each root angle
throughout the sequence. The angle identities allow us to apply
logarithmic Liouville rigidity. Finally, strict cell stability determines
the canonical face maps, and nearest-spacing normalization gives the
asserted rotation.

\medskip\noindent\emph{The whole marked limit and bounded target sets.}
Start with an arbitrary subsequence. Fix a unit source face incident to
$0$ and its intrinsic incenter $o\in\Sig$. Its distance to $\partial Q_n$
tends to infinity. Let $I_n$ be the canonical intrinsic lift given by
Proposition~\ref{prop:reflection-cover}. Write
$I_n(o)=(\beta_n,t_n)$, with $t_n>0$, and set
\[
 M_n(z,t)=\left(\frac{z-\beta_n}{t_n},\frac{t}{t_n}\right),
 \qquad J_n=M_nI_n,
 \qquad X'_n=\frac{X_n-\beta_n}{t_n}.
\]
For $z\in Q_n$, put
\[
 G_n(z)=\frac{g_n(v_n+h_na_nz)-\beta_n}{t_n},
 \qquad z_n(j)=G_n(j)\quad(j\in\Lam\cap Q_n).
\]
By Theorem~\ref{thm:marked-limit}, after extraction the sets
$X'_n\cup\{\infty\}$ converge in spherical Hausdorff distance to a
proper compact set $Y$.
Its convex hull is three-dimensional, and $J_n$ converges on intrinsic
compact sets to an isometry
\[
 J:\Sig\longrightarrow\partial\CH(Y)
\]
onto the entire boundary. Its extension to the filled compactifications
induces on the ideal labels a bijection
$y:\Lam\cup\{\infty\}\to Y$, with $y(\infty)=\infty$ and
$z_n(j)\to y(j)$. Each finite $y(j)$ has a fixed neighborhood containing
exactly $z_n(j)$ from the approximating ideal set for all sufficiently
large $n$. Thus $y|_{\Lam}$ is injective and proper, and $y(\Lam)$ is
locally finite in $\CC$.

We will also use the exact metric-ball conclusion of
Proposition~\ref{prop:reflection-cover}. The normalization $M_n$ is a
hyperbolic isometry. Thus, with
$\Sigma_n=\partial\CH(X'_n\cup\{\infty\})$, for each fixed $R>0$ and
all sufficiently large $n$, the map $J_n$ is an onto isometry from the
closed intrinsic $R$-ball about $o$ in $\Sig$ to the closed intrinsic
$R$-ball about $(0,1)$ in $\Sigma_n$. Both balls carry the distances
inherited from their complete surfaces. In particular every path in
the target ball pulls back under this identification.

For every bounded target disk $B\subset\CC$ there is a finite set
$E_B\subset\Lam$ such that
\begin{equation}\label{eq:microscopic-similarity-finite-labels}
 X'_n\cap B\subset\{z_n(j):j\in E_B\}
 \qquad\text{for all sufficiently large }n.
\end{equation}
Indeed, if this failed, choose an exhaustion of $\Lam$ by finite sets and
points in $X'_n\cap B$ that avoid the corresponding marked sets along an
increasing subsequence. A further subsequence of these points converges
to a finite point of $Y$, hence to $y(j)$ for some fixed $j$. The singleton
neighborhood at $y(j)$ forces the points to equal $z_n(j)$ eventually,
which is a contradiction. This argument includes reflected vertices
outside the initially chosen source half.

\medskip\noindent\emph{Positive limiting scales and logarithmic growth.}
For a fixed label $j$, let $\widehat r_n(j)$ be the canonical scale in
Theorem~\ref{thm:ring} at the physical vertex $v_n+h_na_nj$. Define
\[
 r_n(j)=\frac{\widehat r_n(j)}{t_n},
 \qquad
 d_n(j)=\dist\bigl(z_n(j),X'_n\setminus\{z_n(j)\}\bigr).
\]
Thus lengths, including the canonical scales, are multiplied by
$1/t_n$, exactly as in $M_n$. Every fixed source label is eventually
interior, so its parameters in Theorem~\ref{thm:ring} are $b_j=1$ and
$k_j=6$. That theorem gives
\begin{equation}\label{eq:microscopic-similarity-ring}
 \frac3\pi d_n(j)\le r_n(j)\le3d_n(j),
 \qquad
 |z_n(j)-z_n(k)|\le\sqrt{r_n(j)r_n(k)},
 \qquad
 \frac19\le\frac{r_n(j)}{r_n(k)}\le9,
\end{equation}
where the last two inequalities hold on every fixed original edge $jk$
for all sufficiently large $n$.

The singleton neighborhood at $y(j)$ bounds $d_n(j)$ away from zero.
One fixed neighboring label bounds it above. We may therefore extract
diagonally so that $r_n(j)\to r(j)\in(0,\infty)$ for every $j\in\Lam$.
For distinct labels, and for original neighbors, respectively,
\eqref{eq:microscopic-similarity-ring} implies
\[
 r(j)\le3|y(j)-y(k)|,
 \qquad
 |y(j)-y(k)|\le3\min\{r(j),r(k)\}\quad(jk\text{ an original edge}).
\]
The second bound follows by combining the square-root bound with the
ratio bound. Theorem~\ref{thm:packing-growth} now applies to $y$ and $r$.
With $u_j=\log r(j)$, it gives
\begin{equation}\label{eq:microscopic-similarity-logarithmic}
 |u_j-u_k|\le10^9\log\bigl(N(j,k)+2\bigr)
 \qquad(j,k\in\Lam).
\end{equation}
Theorem~\ref{thm:log-liouville} supplies positive envelopes $F_i(u;x)$,
each computed on a period by a finite family of standard lifts. We next
prove that their angles are all $2\pi$.

\medskip\noindent\emph{The finite-surface root decoration and its angle.}
Fix an arbitrary root $i\in\Lam$. Every fixed label $j$ is eventually
an original interior vertex, where the cover has peripheral degree
one by Proposition~\ref{prop:reflection-cover}. Theorem~\ref{thm:ring}
gives circumference $6/\widehat r_n(j)$ for its original diameter-one
horosphere. The map $M_n$ preserves this circumference and changes the
ambient diameter to $1/t_n$. Restoring diameter one multiplies the
circumference by $t_n$, giving $6/r_n(j)$ on $\Sigma_n$.
Multiply all these ambient diameters by the common factor $r_n(i)$.
The resulting circumference at $j$ is
\[
 6\frac{r_n(i)}{r_n(j)}.
\]
In particular, the root circumference is exactly $6$ for every
sufficiently large $n$.

Normalize a universal covering upper half-plane chart, with coordinates
$(x,s)$, by placing this root horocycle at $s=1$. Its primitive period
is $6$. Fix the chart using a point and orientation on the standard source
root horocycle, with one standard neighboring lift centered at $0$.
The exact source-patch isometries make these choices consistent in $n$.
For every other lifted horocycle $\sigma$, write
$\xi_{n,\sigma}\in\RR$ for its center and $d_{n,\sigma}>0$ for its
Euclidean diameter. Define
\[
 F_{n,i}(x)=\sup_\sigma
 \bigl(d_{n,\sigma}-(x-\xi_{n,\sigma})^2\bigr).
\]
We call a lift \emph{active} on a period if its quadratic attains this
envelope at some point of that period.

These finite-surface envelopes are positive and have finite active
families on a period. To see why the finite correspondence applies, the
diameter-one decoration is the decoration of the actual canonical
Euclidean metric, rescaled by $1/t_n$. The common change by $r_n(i)$
preserves its cells and angles. The doubled sphere has at least three
punctures, so its punctured Euler characteristic is negative. The
decorated-coordinate identification of
Theorem~\ref{thm:finite-uniformization} therefore applies to this finite
quotient. Lifted by Proposition~\ref{prop:reflection-cover}, it
identifies the intrinsic Voronoi dual with the planar Delaunay cells
of $X'_n$. This periodic,
locally finite planar set contains the reflected corner lattice, so its
Delaunay cells are bounded and its vertex stars are finite and full.
Its nondegenerate Euclidean triangles give positive
heights at the meetings of the corresponding root Voronoi arcs.
The root boundary is therefore a positive periodic chain of finitely
many semicircular arcs. Its angle is the full planar angle $2\pi$.
Proposition~\ref{prop:envelope-comparison} gives the exact identity
\begin{equation}\label{eq:microscopic-similarity-finite-angle}
 \int_0^6F_{n,i}(x)^{-1/2}\dd x=2\pi.
\end{equation}
This applies to the finite canonical metric and its cover.

For a fixed lift $\sigma$ of a standard source horocycle, let
$\xi_\sigma$ and $d^0_\sigma$ be its center and diameter in this root
chart, and let $j(\sigma)$ be its label. A finite-length arc from the
root section to that lift lies in a fixed intrinsic ball. The exact
ball identification therefore transports the arc and the lift to
$\Sigma_n$ for all sufficiently large $n$. Circumference and diameter
scale by the same factor in a fixed chart, so
\begin{equation}\label{eq:microscopic-similarity-diameter-change}
 \xi_{n,\sigma}=\xi_\sigma,
 \qquad
 d_{n,\sigma}=\frac{r_n(i)}{r_n(j(\sigma))}\,d^0_\sigma
 \longrightarrow e^{u_i-u_{j(\sigma)}}d^0_\sigma.
\end{equation}
It remains to show that one fixed finite family contains every active
lift of $F_{n,i}$.

\medskip\noindent\emph{All active lifts have bounded endpoint labels.}
First we bound the diameter of every lift, allowing arbitrary homotopy
classes. Suppose that its target endpoint is $z\in X'_n$ with
$z\ne z_n(i)$. Then
\begin{equation}\label{eq:microscopic-similarity-ambient-bound}
 d_{n,\sigma}\le\frac{r_n(i)^2}{|z-z_n(i)|^2}.
\end{equation}
For the proof, shrink both endpoint decorations by a factor
$\varepsilon>0$ until their terminal cusp sections are embedded and
disjoint. In the root chart, the truncated intrinsic ideal geodesic has
length
\[
 2\log\frac1{\varepsilon\sqrt{d_{n,\sigma}}}.
\]
Its image on $\Sigma_n$ is an ambient hyperbolic path between
horospheres of common diameter $\varepsilon r_n(i)$ at $z_n(i)$ and
$z$. The distance between those horospheres is
\[
 2\log\frac{|z-z_n(i)|}{\varepsilon r_n(i)}.
\]
The path length is at least this distance. Cancelling $\varepsilon$
proves \eqref{eq:microscopic-similarity-ambient-bound}. The comparison
does not require the ideal arc to be a Delaunay edge or its projection
to be simple.

Since $r_n(i)\le3d_n(i)$, every lift with a different target endpoint
has $d_{n,\sigma}\le9$. A different lift of the root cusp itself has
the root's standard circumference $6$. The complete standard source
root cusp is embedded in its six-face source patch, so its distinct
lifts have disjoint interiors. Disjointness from the chosen root
horoball $s>1$ gives $d_{n,\sigma}\le1$ in this case. Thus
\[
 d_{n,\sigma}\le9\quad\text{for every lift }\sigma.
\]
Positivity of $F_{n,i}$ now puts the center of every active lift on
$[0,6]$ in $[-3,9]$.

Choose $c_0>0$ such that $F_i(u;x)\ge c_0$ on $[0,6]$, using
Theorem~\ref{thm:log-liouville}. The finite standard family computing
this limiting envelope persists by
\eqref{eq:microscopic-similarity-diameter-change}. Its quadratics
converge uniformly, and hence
\begin{equation}\label{eq:microscopic-similarity-positive-lower}
 F_{n,i}(x)\ge c_0/2
 \qquad(0\le x\le6)
\end{equation}
for all sufficiently large $n$. Every active lift then has
$d_{n,\sigma}\ge c_0/2$. By
\eqref{eq:microscopic-similarity-ambient-bound}, its non-root target
endpoint satisfies
\[
 |z-z_n(i)|\le r_n(i)\sqrt{2/c_0}.
\]
Both $z_n(i)$ and $r_n(i)$ are bounded. The finite-label conclusion
\eqref{eq:microscopic-similarity-finite-labels} therefore places every
active endpoint in one fixed finite set $E\subset\Lam$, enlarged to
include $i$.

\medskip\noindent\emph{A common finite family of marked lifts.}
For $j\in E$ the ratios $r_n(j)/r_n(i)$ stay in a compact subinterval
of $(0,\infty)$. Return the endpoint horocycle of an active lift to
standard source circumference $6$. Its center is unchanged, and its
diameter becomes
\[
 d^0_{n,\sigma}=\frac{r_n(j)}{r_n(i)}d_{n,\sigma}.
\]
This identity rescales every lift of the cusp and therefore applies
before its homotopy class has been identified. There is a fixed
$\eta>0$ such that $d^0_{n,\sigma}\ge\eta$ for all these lifts.
The standard cusp neighborhoods at the finitely many labels in $E$
are mutually disjoint in the source patch and hence in $\Sigma_n$.
Their lifts are disjoint from the root horoball. Consequently
\[
 \eta\le d^0_{n,\sigma}\le1,
 \qquad -3\le\xi_{n,\sigma}\le9.
\]

The vertical segment from $(\xi_{n,\sigma},1)$ to
$(\xi_{n,\sigma},d^0_{n,\sigma})$ has hyperbolic length
\[
 \log\frac1{d^0_{n,\sigma}}\le\log\frac1\eta.
\]
Its initial point projects to the compact standard root horocycle.
The entire projected segment therefore belongs to an intrinsic ball
about $J_n(o)$ of radius independent of $n$ and of the active lift.
The exact entire-ball identification of
Proposition~\ref{prop:reflection-cover} pulls this segment back into
one fixed source patch. Develop the pulled-back segment starting from
the fixed lifted root section. It reaches a standard source horocycle
with the same center $\xi_{n,\sigma}$ and diameter $d^0_{n,\sigma}$.
The segment identifies the lift, including its homotopy class.

We have thus placed every active candidate in the same universal cover
of $\Sig$. In that cover the standard open horoballs are pairwise
disjoint. Those with centers in $[-3,9]$ and diameters in $[\eta,1]$
lie in the rectangle $[-7/2,19/2]\times[0,1]$, and each has Euclidean
area at least $\pi\eta^2/4$. Only finitely many exist. This gives one
finite list of standard lifts containing all active finite-surface
lifts for every sufficiently large $n$.

Enlarge this list by the finite family computing $F_i(u;x)$. Every
member of the enlarged list persists, and all its coefficients
converge by \eqref{eq:microscopic-similarity-diameter-change}. All
its limiting quadratics belong to the family defining $F_i(u;x)$.
Both envelopes are therefore maxima over this one finite list for
all sufficiently large $n$, and
\[
 F_{n,i}\longrightarrow F_i(u;\cdot)
 \quad\text{uniformly on }[0,6].
\]
The common positive lower bound
\eqref{eq:microscopic-similarity-positive-lower} permits passage to
the reciprocal square roots in
\eqref{eq:microscopic-similarity-finite-angle}. We obtain
\begin{equation}\label{eq:microscopic-similarity-limit-angle}
 \theta_i(u)=\int_0^6F_i(u;x)^{-1/2}\dd x=2\pi.
\end{equation}
The root $i$ was arbitrary. The argument uses the already convergent
scale coordinates and requires no further extraction depending on $i$.
Thus \eqref{eq:microscopic-similarity-limit-angle} holds at every
lattice vertex.

\medskip\noindent\emph{Constant scales and strict marked cell stability.}
Apply Theorem~\ref{thm:log-liouville} to
\eqref{eq:microscopic-similarity-logarithmic} and
\eqref{eq:microscopic-similarity-limit-angle}. It gives $u$ constant.
Write $r(j)=r>0$ for its common exponential. The root decorations
are now standard in the limit. Their neighboring lifts have diameter
$1$ and centers at consecutive integers, with six such lifts modulo
the root period. Their envelope is
\[
 F_i(u;x)=1-\dist(x,\ZZ)^2\ge\frac34.
\]

There is a strict gap to every other standard lift. If its center is
$\xi$ and its diameter is $d^0$, disjointness from the unit-diameter
neighbor horoball centered at each $k\in\ZZ$ gives
\[
 (\xi-k)^2+\left(\frac{d^0-1}{2}\right)^2
 \ge\left(\frac{d^0+1}{2}\right)^2,
 \qquad
 d^0\le\dist(\xi,\ZZ)^2\le\frac14.
\]
Its quadratic is therefore at most $1/4$ everywhere. Each original
neighbor dominates on the interior of its unit interval. At an
endpoint of that interval only the two consecutive neighbor
quadratics meet, and their difference has nonzero slope.

The common finite list just proved makes these properties stable under
the coefficient convergence. The non-neighbor quadratics remain
strictly below the envelope, and each consecutive pair has one
transverse switch near its original switch. Thus the same six
genuine neighbors, in the same cyclic order, give the same triangular
cells for all sufficiently large $n$. The limiting equilateral
opposite-angle sum along each edge is $2\pi/3<\pi$, so these cells
are strictly Delaunay. Because the list consists of marked lifts,
the cells have the original ideal arc classes as well as the original
endpoint labels. Applying this conclusion at the finitely many
vertices of any fixed original face patch proves its persistence as
a marked target Delaunay patch.

\medskip\noindent\emph{The similarity and the canonical closed-face maps.}
On every persistent original edge $jk$, its target chord length equals
its intrinsic lambda length for diameter-one horospheres. Hence
Theorem~\ref{thm:ring} gives
\[
 |z_n(j)-z_n(k)|=\sqrt{r_n(j)r_n(k)}\longrightarrow r.
\]
The target triangles have the source orientation, since the canonical
intrinsic map and the planar projections preserve orientation. Put
$b=y(0)$ and $c=y(1)-y(0)$. Then $|c|=r$, and the positively oriented
equilateral triangle at $0,1,\om$ gives $y(\om)=b+c\om$.
Across a shared edge, the third vertex of the next equilateral
triangle is uniquely determined on the required side. Propagation
through the connected source triangulation gives
\begin{equation}\label{eq:microscopic-similarity-affine}
 y(j)=b+cj\quad(j\in\Lam),
 \qquad Y=(b+c\Lam)\cup\{\infty\}.
\end{equation}
Let $\Psi(z)=b+cz$ and let
$\widehat\Psi(z,t)=(b+cz,|c|t)$ be its hyperbolic extension.

Fix an original closed Euclidean face $\tau$, and let
$\overline\tau^{\,H}$ be the corresponding closed ideal face of
$\Sig$, with its three ideal vertices included. For all sufficiently
large $n$ its marked image is the corresponding target ideal face.
On this face $J_n$ is the intrinsic hyperbolic isometry determined by
the three ordered ideal vertices. The target triples converge to
the nondegenerate triple prescribed by
\eqref{eq:microscopic-similarity-affine}.

The unique three-point M\"obius transformation carrying the source
triple to the target triple, and its hyperbolic extension, depend
continuously on those distinct vertices. Their restrictions to
$\overline\tau^{\,H}$ are precisely these face isometries. They
converge to $\widehat\Psi$ on the closed ideal face. The target
circumcircles converge to the nondegenerate circumcircle of
$\Psi(\tau)$, so the supporting hemispheres and all these closed
faces stay in a bounded part of the closed upper half-space, away
from its ideal point at infinity. The convergence is consequently
uniform there in Euclidean coordinates.

Write $p_{\mathrm{vert}}$ for vertical projection. Its restriction to a closed
ideal face, including its vertices, is a homeomorphism onto the
closed Euclidean face. By the definition of the canonical map,
\[
 G_n|_\tau
 =p_{\mathrm{vert}}\circ J_n\circ(p_{\mathrm{vert}}|_{\overline\tau^{\,H}})^{-1}.
\]
The preceding convergence therefore gives
$G_n\to\Psi$ uniformly on $\tau$, including its edges and vertices.
Every compact source set is contained in a finite union of original
faces, so this convergence is uniform on every compact subset of
$\CC$. On each intrinsic face interior, the whole-boundary limit
$J$ equals $\widehat\Psi$. These interiors are dense in $\Sig$;
continuity gives $J=\widehat\Psi$ on the entire surface.

\medskip\noindent\emph{Nearest-spacing normalization.}
The nearest spacing in the incenter normalization is
\[
 \widehat\delta_n=\frac{\delta_n}{t_n}
 =\dist\bigl(z_n(0),X'_n\setminus\{z_n(0)\}\bigr).
\]
The fixed neighbor $1$ gives $\limsup_n\widehat\delta_n\le|c|$.
Choose a nearest point for each $n$. These points remain in a bounded
target disk. By \eqref{eq:microscopic-similarity-finite-labels}, a
subsequence realizing the lower limit has a further subsequence on
which that point is $z_n(j)$ for one fixed $j\ne0$. Equation
\eqref{eq:microscopic-similarity-affine} then gives a limiting
distance $|c|\,|j|\ge|c|$. Thus
$\widehat\delta_n\to|c|>0$.

It follows that
\[
 A_n(z)=\frac{G_n(z)-z_n(0)}{\widehat\delta_n}
 \longrightarrow\frac{c}{|c|}z
 \quad\text{uniformly on every compact source set}.
\]
Choose $\theta\in\RR$ with $c/|c|=e^{i\theta}$. The similarities
$z\mapsto(z-z_n(0))/\widehat\delta_n$ converge to the nondegenerate
similarity $z\mapsto(z-b)/|c|$, also uniformly on the ideal sphere.
Applying them to the spherical Hausdorff convergence of
$X'_n\cup\{\infty\}$ proves the asserted convergence of
$Z_n\cup\{\infty\}$.

For any compact target set $K$, its inverse images under these
similarities lie in one bounded set. Equation
\eqref{eq:microscopic-similarity-finite-labels} therefore gives the
required fixed finite set of labels for $Z_n\cap K$. Similarities
preserve the marked Delaunay cells, so the persistence already proved
also holds in this normalization. Composing $J=\widehat\Psi$ with
the limiting normalization gives the hyperbolic extension of
$z\mapsto e^{i\theta}z$ on the whole intrinsic boundary. These
observations give, respectively, the finite-label conclusion, the
marked patch conclusion, and the whole-boundary conclusion; the
closed-face argument gives the claimed convergence of the maps.
Every extraction began with an arbitrary subsequence.
\end{proof}

\section{Proof of uniform convergence on the closed triangle}\label{sec:08-main-proof}

We keep the notation for the canonical maps and the closed reference
triangle. We combine compactness on that triangle with microscopic
similarity to identify every subsequential limit. The convergence
obtained includes all three sides and their marked corners.

\begin{theorem}\label{thm:main}
Let $\Omega\subset\CC$ be a bounded Jordan domain with three distinct
positively ordered marked boundary points $p,q,r$. Let
$(P_n,T_n,h_n;p_n,q_n,r_n)$ be any sequence of marked equilateral disks
such that $h_n\to0$, their Jordan boundaries converge uniformly to
$\partial\Omega$, and $(p_n,q_n,r_n)\to(p,q,r)$.
Then the canonical discrete uniformization maps
$f_n:\Tri\to P_n$ converge uniformly on the entire closed triangle
$\Tri$ to the normalized Riemann map
$\phi:\Tri\to\overline\Omega$ with the prescribed corner values.
\end{theorem}

\begin{proof}[Proof of Theorem~\ref{thm:main}]
We first use microscopic similarity to control the signed shapes of all
original faces in a physical compact subset of $\Omega$. We then compare
the actual inverse maps with their affine interpolants, use degree to
bound the affine area and energy, and prove that every homeomorphic limit
is conformal. Closed-triangle compactness and normalization will identify
the limit of the full sequence. Set $g_n=f_n^{-1}$.

\medskip\noindent
\emph{Physical compact subsets lie deep in the source triangulations.}
Choose homeomorphic parametrizations
$\gamma_n:S^1\to\partial P_n$ and
$\gamma:S^1\to\partial\Omega$ with $\gamma_n\to\gamma$ uniformly,
taking $\gamma$ positively oriented. Put
$\eta_n=\sup_{S^1}|\gamma_n-\gamma|$.
Fix a nonempty compact set $K\Subset\Omega$, and set
\[
 d_K=\dist(K,\partial\Omega),
 \qquad
 K^+=\{z\in\CC:\dist(z,K)\le d_K/4\}.
\]
Then $d_K>0$ and $K^+\Subset\Omega$. Once $\eta_n<d_K/4$, the
straight homotopy from $\gamma$ to $\gamma_n$ avoids every point of
$K^+$. The winding number about each such point remains one. The Jordan
curve $\partial P_n$ therefore encloses $K^+$, so
$K^+\subset\operatorname{int}P_n$. Moreover,
\[
 \dist(K,\partial P_n)\ge d_K-\eta_n.
\]
An original face $\tau$ meeting $K$ has diameter $h_n$. For all
sufficiently large $n$, it follows that
\begin{equation}\label{eq:main-proof-face-depth}
 \tau\subset K^+,
 \qquad
 \dist(\tau,\partial P_n)
 \ge d_K-\eta_n-h_n\ge d_K/2.
\end{equation}
These conclusions hold uniformly over all faces meeting $K$.

\medskip\noindent
\emph{The signed face ratios converge uniformly.}
For any counterclockwise ordering $(z_0,z_1,z_2)$ of an original face,
define
\[
 R_n(z_0,z_1,z_2)
 =\frac{g_n(z_2)-g_n(z_0)}{g_n(z_1)-g_n(z_0)}.
\]
Its denominator is nonzero because $g_n$ is a homeomorphism. The source
ratio $(z_2-z_0)/(z_1-z_0)$ equals $\om$. In the following maximum,
$\tau$ ranges over all original faces meeting $K$, and
$(z_0,z_1,z_2)$ ranges over all counterclockwise orderings of its vertices.
We claim that
\begin{equation}\label{eq:main-proof-face-shape}
 \epsilon_n(K)
 :=\max_{\tau,(z_0,z_1,z_2)}
       |R_n(z_0,z_1,z_2)-\om|
 \longrightarrow0.
\end{equation}

Suppose this fails. After restricting to a subsequence, choose an ordered
face $(z_{0,n},z_{1,n},z_{2,n})$ for each $n$ such that
\[
 |R_n(z_{0,n},z_{1,n},z_{2,n})-\om|\ge\varepsilon_0
\]
for one fixed $\varepsilon_0>0$. Set
\[
 v_n=z_{0,n},
 \qquad
 \zeta_n=\frac{z_{1,n}-z_{0,n}}{h_n},
 \qquad
 Q_n=\frac{P_n-v_n}{h_n\zeta_n}.
\]
Here $|\zeta_n|=1$, and the ordered face becomes $(0,1,\om)$.
Across a shared edge, the third vertex of an equilateral face is
determined by the side of that edge on which it lies. Propagating through
the connected face adjacency of the disk shows that $Q_n$ is a subcomplex
of the unit triangular lattice. By
\eqref{eq:main-proof-face-depth}, $v_n$ is an interior vertex and
\[
 \overline B\left(0,\frac{d_K}{3h_n}\right)\subset Q_n.
\]
The contained radii tend to infinity, so
Theorem~\ref{thm:microscopic} applies to this sequence of moving roots.

Let $X_n$ be the reflected target vertex set, and use the normalization
of that theorem:
\[
 w_n=g_n(v_n),
 \qquad
 \delta_n=\dist(w_n,X_n\setminus\{w_n\}),
 \qquad
 A_n(z)=\frac{g_n(v_n+h_n\zeta_n z)-w_n}{\delta_n}.
\]
On a further subsequence, for some $\theta\in\RR$,
\[
 A_n(0)=0,
 \qquad A_n(1)\longrightarrow e^{i\theta},
 \qquad A_n(\om)\longrightarrow e^{i\theta}\om.
\]
The limiting denominator has modulus one. Hence
\[
 R_n(z_{0,n},z_{1,n},z_{2,n})
 =\frac{A_n(\om)}{A_n(1)}\longrightarrow\om,
\]
contrary to the choice of the faces. This proves
\eqref{eq:main-proof-face-shape}. The offending ordering was part of the
choice, so the conclusion is uniform over every counterclockwise
ordering as well as every face meeting $K$.

\medskip\noindent
\emph{The actual inverses and their affine interpolants have the same limit.}
Start with any subsequence of the original sequence.
Theorem~\ref{thm:closed-compactness} gives a further subsequence, still
indexed by $n$, and a marked homeomorphism
$F:\Tri\to\overline\Omega$ such that
\[
 f_n\longrightarrow F\quad\hbox{uniformly on }\Tri.
\]
Set $G=F^{-1}$. The same theorem gives
\begin{equation}\label{eq:main-proof-inverse-limit}
 g_n\longrightarrow G
 \quad\hbox{uniformly on every compact subset of }\Omega.
\end{equation}
The positive cyclic order of the three boundary marks implies that $F$
and $G$ preserve orientation.

Define $L_n:P_n\to\Tri$ by setting $L_n(v)=g_n(v)$ at every original
vertex and extending by real affine interpolation on each original
closed face. The affine traces agree on shared edges, so $L_n$ is
continuous. Convexity of $\Tri$ gives $L_n(P_n)\subset\Tri$.
For $x$ in a face with vertices $z_0,z_1,z_2$, let
$t_0,t_1,t_2\ge0$ be its barycentric coordinates. Thus
\[
 \sum_{j=0}^2t_j=1,
 \qquad L_n(x)=\sum_{j=0}^2t_jg_n(z_j),
 \qquad
 |L_n(x)-g_n(x)|\le\max_{0\le j\le2}|g_n(z_j)-g_n(x)|.
\]
For a compact $K$ as above, the faces meeting $K$ lie in $K^+$ for large
$n$. Consequently,
\[
 \sup_K|L_n-g_n|
 \le 2\sup_{K^+}|g_n-G|
    +\sup_{\substack{x,y\in K^+\\|x-y|\le h_n}}|G(x)-G(y)|
 \longrightarrow0.
\]
Here the first term tends to zero by
\eqref{eq:main-proof-inverse-limit}, and the second by uniform continuity
of $G$ on $K^+$. We have proved
\begin{equation}\label{eq:main-proof-affine-limit}
 L_n\longrightarrow G\quad\hbox{locally uniformly on }\Omega.
\end{equation}

\medskip\noindent
\emph{The affine antiholomorphic derivatives are small relative to the
holomorphic derivatives.}
Fix concentric open disks $U',U$ with
$\overline{U'}\subset U$ and $\overline U\Subset\Omega$.
For large $n$, all faces meeting $\overline U$ satisfy
$\epsilon_n(\overline U)<\sqrt3/2$. On such a face, apply direct
similarities in the source and target to send its ordered vertices and
their values to
\[
 (0,1,\om)\longmapsto(0,1,R),
 \qquad R=R_n(z_0,z_1,z_2).
\]
The normalized affine map is $z\mapsto\alpha z+\beta\overline z$, with
\[
 \alpha+\beta=1,
 \qquad \alpha\om+\beta\overline\om=R,
 \qquad
 \alpha=\frac{R-\overline\om}{\om-\overline\om},
 \qquad
 \beta=\frac{\om-R}{\om-\overline\om}.
\]
Since $|\om-\overline\om|=\sqrt3$, we obtain
\[
 \frac{|\beta|}{|\alpha|}
 \le\frac{|R-\om|}{\sqrt3-|R-\om|}
 \le k_n,
 \qquad
 k_n:=\frac{\epsilon_n(\overline U)}
              {\sqrt3-\epsilon_n(\overline U)}\longrightarrow0.
\]
For large $n$ we have $k_n<1$ and $\alpha\ne0$. Direct similarities
preserve the absolute ratio of antiholomorphic to holomorphic
derivatives. Therefore, almost everywhere on $U$,
\begin{equation}\label{eq:main-proof-affine-distortion}
 |\bpartial L_n|\le k_n|\partial L_n|,
 \qquad
 J_{L_n}=|\partial L_n|^2-|\bpartial L_n|^2>0.
\end{equation}
For each fixed $n$, the continuous map $L_n$ on the finite triangulation
belongs to $W^{1,2}(U)$. Its weak derivatives are the facewise affine
derivatives: integration by parts on the faces cancels the internal-edge
terms because their traces agree.

\medskip\noindent
\emph{Degree bounds the area and the energy on $U'$.}
The image $G(U)$ is a Jordan domain with boundary $G(\partial U)$.
Set
\[
 d_*:=\dist\bigl(G(\overline{U'}),G(\partial U)\bigr)>0,
 \qquad
 e_n:=\sup_{\overline U}|L_n-G|\longrightarrow0.
\]
Take $n$ large enough that $e_n<d_*/3$. For any
$y\in L_n(\overline{U'})$, choose $x\in\overline{U'}$ with
$y=L_n(x)$. The segment from $G(x)$ to $y$ stays within distance $e_n$
of $G(x)$ and avoids $G(\partial U)$. It follows that $y\in G(U)$.
For every $z\in\partial U$ and $t\in[0,1]$, we also have
\[
 \bigl|(1-t)G(z)+tL_n(z)-y\bigr|
 \ge |G(z)-G(x)|-t|L_n(z)-G(z)|-|L_n(x)-G(x)|
 \ge d_*-2e_n>0.
\]
Thus the straight boundary homotopy avoids every such $y$. Invariance of
winding number and the positive orientation of $G$ give
\begin{equation}\label{eq:main-proof-degree}
 \deg(L_n,U,y)=\deg(G,U,y)=1
 \qquad\bigl(y\in L_n(\overline{U'})\bigr).
\end{equation}
The estimate is uniform in $x$, so this assertion includes target points
depending on $n$. It uses only the winding number of the continuous
boundary loop $L_n|_{\partial U}$.

Let $\mathcal E_n$ be the union of the images under $L_n$ of all original
mesh edges. This is a finite union of line segments and has planar area
zero. If $y\in L_n(\overline{U'})\setminus\mathcal E_n$, each preimage
in $U$ lies in the interior of an original face. There are finitely many
such preimages, and each has local degree $+1$ by
\eqref{eq:main-proof-affine-distortion}. Equation
\eqref{eq:main-proof-degree} counts their sum, so there is exactly one
preimage in $U$. In particular, the multiplicity
\[
 m_n(y):=\#\{x\in U':L_n(x)=y\}
\]
is at most one for almost every $y\in\CC$, and it is zero outside
$\Tri$. Applying affine change of variables on each measurable set
$U'\cap\operatorname{int}\tau$ and summing over the original faces
therefore gives
\begin{equation}\label{eq:main-proof-area}
 \int_{U'}J_{L_n}\dd A
 =\int_{\CC}m_n(y)\dd A(y)
 \le\Area(\Tri)=\frac{\sqrt3}{4}.
\end{equation}

Combining \eqref{eq:main-proof-affine-distortion} and
\eqref{eq:main-proof-area}, we obtain
\begin{equation}\label{eq:main-proof-energy}
 \begin{aligned}
 \int_{U'}|\partial L_n|^2\dd A
 &\le\frac{\Area(\Tri)}{1-k_n^2},\\
 \int_{U'}|\bpartial L_n|^2\dd A
 &\le\frac{k_n^2}{1-k_n^2}\Area(\Tri)\longrightarrow0.
 \end{aligned}
\end{equation}
The identity $|DL_n|^2=2(|\partial L_n|^2+|\bpartial L_n|^2)$ also yields
\[
 \int_{U'}|DL_n|^2\dd A
 \le\frac{2(1+k_n^2)}{1-k_n^2}\Area(\Tri).
\]

\medskip\noindent
\emph{The limit is conformal and determines the full sequence.}
For $\psi\in C_c^\infty(U')$, local uniform convergence in
\eqref{eq:main-proof-affine-limit}, integration by parts, and
\eqref{eq:main-proof-energy} give
\[
 \int_{U'}G\,\bpartial\psi\dd A
 =\lim_{n\to\infty}\int_{U'}L_n\,\bpartial\psi\dd A
 =-\lim_{n\to\infty}\int_{U'}(\bpartial L_n)\psi\dd A=0.
\]
The last limit is zero by the Cauchy--Schwarz inequality. Every disk
compactly contained in $\Omega$ can serve as $U'$, so
$\bpartial G=0$ in distributions throughout $\Omega$.
On a smaller disk, convolving $G$ with smooth approximate identities
produces smooth solutions of the Cauchy--Riemann equation. These are
holomorphic and converge locally uniformly to the continuous map $G$.
Passing to the Cauchy integral formula shows that $G$ is holomorphic.

An injective holomorphic map has nonvanishing derivative. Thus
$F=G^{-1}$ is conformal on $\operatorname{int}\Tri$. Its continuous
extension has the prescribed three corner values, so uniqueness of the
normalized Riemann map gives $F=\phi$ on the entire closed triangle.
Here microscopic similarity and the degree estimate establish interior
conformality, while Theorem~\ref{thm:closed-compactness} supplies the
homeomorphic limit and convergence up to all boundary points.

We began with an arbitrary subsequence. Each such subsequence has a
further subsequence converging uniformly on $\Tri$ to $\phi$.
If the full sequence did not converge uniformly, some subsequence would
satisfy
\[
 \sup_{z\in\Tri}|f_n(z)-\phi(z)|\ge\varepsilon_1
\]
for one fixed $\varepsilon_1>0$. Its further uniformly convergent
subsequence contradicts this inequality. Hence $f_n\to\phi$ uniformly
on the entire closed triangle.
\end{proof}

\section*{Acknowledgements}
The first author is supported by NSFC (No.~12571055), NSFHN
(No.~2022JJ10059) and Scientific Research Program of NUDT (No.~JS2023-01).

\section*{Declaration on the use of generative AI}
The authors used a generative AI assistant for language editing and
typographical consistency checks during the preparation of this manuscript.
All mathematical results, including the construction, proofs, and their
verification, were produced by the authors. The authors have reviewed the
final manuscript in its entirety and take full responsibility for its
content.

\vspace*{1.5em}
\noindent
Chunlei Liu, \texttt{7357@cnu.edu.cn}\\
School of Mathematical Sciences, Capital Normal University,
Beijing 100048, P.~R.~China.

\vspace*{0.6em}
\noindent
Aijin Lin, \texttt{linaijin@nudt.edu.cn}\\
College of Science, National University of Defense Technology,
Changsha 410073, P.~R.~China.

\end{document}